\documentclass[a4paper,12pt]{article}
\usepackage{amsmath,amsfonts,amssymb,amscd}
\usepackage{indentfirst,graphicx,epstopdf}
\usepackage{graphicx,psfrag,ifpdf}
\usepackage{caption}
\usepackage{amsthm, amsfonts, color}
\usepackage[bookmarksnumbered, plainpages]{hyperref}
\usepackage[numbers,sort&compress]{natbib}
\input{epsf}

\newcommand\tabcaption{\def\@captype{table}\caption}

\date{}\makeatother
\newtheorem{theorem}{Theorem}[section]
\newtheorem{pro}{Proposition}[section]
\newtheorem{defi}{Definition}[section]
\newtheorem{lemma}{Lemma}[section]
\newtheorem{coro}{Corollary}[section]

\newtheorem{conjecture}{Conjecture}
\newtheorem{claim}{Claim}
\newtheorem{fact}{Fact}

\title{Proper vertex-pancyclicity of edge-colored complete graphs without joint monochromatic triangles\thanks{Supported by NSFC No.11871034 and 11531011.}}
\author{Xiaozheng Chen, Xueliang Li\\
{\small Center for Combinatorics and LPMC}\\
{\small Nankai University, Tianjin 300071, China}\\
{\small Email: chen\_xiaozheng@163.com, lxl@nankai.edu.cn}
}

\begin{document}

\maketitle

\begin{abstract}
In an edge-colored graph $(G,c)$, let $d^c(v)$ denote the number of colors on the edges incident
with a vertex $v$ of $G$ and $\delta^c(G)$ denote the minimum value of $d^c(v)$ over all vertices $v\in V(G)$.
A cycle of $(G,c)$ is called proper if any two adjacent edges of the cycle have distinct colors.
An edge-colored graph $(G,c)$ on $n\geq 3$ vertices is called properly vertex-pancyclic
if each vertex of $(G,c)$ is contained in a proper cycle of length $\ell$ for every $\ell$ with $3 \le \ell \le n$.
Fujita and Magnant conjectured that
every edge-colored complete graph on $n\geq 3$ vertices with $\delta^c(G)\geq \frac{n+1}{2}$
is properly vertex-pancyclic.
Chen, Huang and Yuan partially solve this conjecture by adding an extra condition that
$(G,c)$ does not contain any monochromatic triangle.
In this paper, we show that this conjecture is true if
the edge-colored complete graph contain no joint monochromatic triangles. \\

\noindent{\bf Keywords:} edge-colored graph, proper cycle, color degree, properly vertex-pancyclicity
\noindent{\bf AMS subject classification 2010:} 05C38, 05C15.
\end{abstract}

\section{Introduction}

Let $G$ be a simple graph and let $c$ be an edge-coloring of $G$. We call $(G,c)$ an \emph{edge-colored graph}.
A subgraph in an edge-colored graph is called \emph{proper} if any two adjacent edges in the subgraph are colored by distinct colors.
Rainbow subgraphs and monochromatic subgraphs are two popular concepts related to proper subgraphs. A subgraph in an edge-colored graph
is called \emph{rainbow} if all the edges in the subgraph are colored by distinct colors, and \emph{monochromatic} if all the edges
in the subgraph are colored by the same color.

In this paper, we mainly consider proper cycles in an edge-colored graph. Once a cycle is found and denoted by the cyclic arrangement
of its vertices such that two vertices are adjacent if they are consecutive in the sequence and nonadjacent otherwise in an edge-colored graph,
one can easily checked whether it is proper or not. So, we often omit the checking process in the following.

In an edge-colored graph $(G,c)$, let $d^c_G(v)$ denote the number of colors on the edges incident with a
vertex $v$ of $G$ and let $\delta^c(G)$ denote the minimum value of $d^c_G(v)$ over all vertices $v\in V(G)$.
When no confusion occurs, we use $d^c(v)$ instead of $d^c_G(v)$. The \emph{length} of a path or a cycle is the number of its edges.
 Let $\Delta^{mon}(K^c_n)$
denote the maximum number of edges of the same color incident with a
vertex of $K^c_n$.

An edge-colored graph $(G,c)$ is called {\it properly Hamiltonian} if it contains a properly colored Hamiltonian cycle.
An edge-colored graph $(G,c)$ is  called {\it properly vertex-pancyclic} if every vertex of the graph is contained in
a proper cycle of each length $\ell$ for every $\ell$ with $3 \le \ell \le n$.

In 1952, Dirac \cite{Dirac} obtained a classical theorem that if $\delta(G)\geq \frac{n}{2}$, then $G$ is Hamiltonian. Inspired by this work, there have appeared lots of results and problems on the existence of proper cycles in different types of edge-colored graphs.
In 1976, Bollobas and
Erdos \cite{BE} conjectured that every $K^c_n$ with $\Delta^{mon}(K^c_n)<\lfloor\frac{n}{2}\rfloor$ contains a properly colored
Hamiltonian cycle.
The author in \cite{A} showed that for any $\varepsilon > 0$, there exists an integer $n_ 0$
such that every $K^ c_n$ with $\Delta^{mon}(K^c_n)<(\frac{n}{2} - \varepsilon )n$ and $n\geq n_ 0$ contains a properly colored
Hamiltonian cycle, which implies a result obtained by Alon and Gutin \cite{AG}
that for every $\varepsilon > 0$ and $n > n_0(\varepsilon)$,
any complete graph $K_n$ on $n$ vertices whose edges are colored so that no vertex is incident with more than $(1 -1/\sqrt{2}- \varepsilon)n$ edges of the same color, contains a Hamiltonian cycle in which adjacent edges have distinct colors. Moreover, for every $k$ between 3 and $n$, any such $K_n$ contains a cycle of length $k$ in which adjacent edges have distinct colors.

\section{Preliminaries}

Fujita and Magnant in \cite{FM} posed the following conjecture.

\begin{conjecture} [\cite{FM}] \label{conj1}
Let $(G, c)$ be an edge-colored graph on $n\geq 3$ vertices. If $\delta^c(G)\geq \frac{n+1}{2}$, then $G$ is properly Hamiltonian.
\end{conjecture}

They showed there that the condition $\delta^c(G)\geq \frac{n+1}{2}$ in Conjecture \ref{conj1} is sharp by constructing an example in \cite{FM}
Then, they further posed the following conjecture.
\begin{conjecture}[\cite{FM}] \label{conj2}
Let $(G,c)$ be an edge-colored complete graph on $n\geq 3$ vertices.
If $\delta^c(G)\geq \frac{n+1}{2}$, then $G$ is properly vertex-pancyclic.
\end{conjecture}

Chen, Huang and Yuan partially solved the conjecture
by adding a condition that $(G,c)$ does not contain any monochromatic triangle.
\begin{theorem}\label{th1} \cite{CHY}
Let $(G,c)$ be an edge-colored complete graph on $n\geq 3$ vertices such that $\delta^c(G)\geq \frac{n+1}{2}$.
If $(G, c)$ contains no monochromatic triangles, then $(G,c)$ is properly vertex-pancyclic.
\end{theorem}

They employed a term named as ``follower vertex"; see the following definition.
\begin{defi} Let $C=v_1v_2\ldots v_\ell v_1$ be a cycle in an edge-colored graph $(G,c)$
and let $v_{\ell+1}=v_1$ and $v_0=v_\ell$.
We say that a vertex $x\in V(G)\setminus V(C)$ follows the colors of $C$ increasingly
if $c(xv_i)=c(v_iv_{i+1})$ for all $i=1,2,\ldots, \ell$,
and a vertex $x\in V(G)\setminus V(C)$ follows the colors of $C$ decreasingly
if $c(xv_i)=c(v_iv_{i-1})$ for all $i=1,2,\ldots, \ell$.
In either of these cases, we say that the vertex $x\in V(G)\setminus V(C)$ ``follows" the colors of $C$
and it is also called a {\it follower vertex}.
\end{defi}

In the proof of Theorem \ref{th1}, they showed two claims which is stated as follows since we will use them
later in our proof of Lemma \ref{lemma2}.
\begin{claim}\label{claim1}
Suppose there is a cycle $C$ of length $\ell$ containing $v_1$, but no proper cycle of length $\ell+1$
containing $v_1$ in $(G,c)$, and suppose there is no monochromatic triangle containing two vertices in $V(C)$
and a vertex in $V(G)\setminus V(C)$.
If there are two vertices which follow the colors of $C$ in different directions,
then $c(v_{i}v_{i+1})= c(v_{i+2}v_{i+3})$ for all indices $i$ with $1\leq i\leq \ell-1$,
which implies that $C$ is an even cycle with two colors appearing alternatively on $C$.
\end{claim}

\begin{claim}\label{claim2}
Suppose there is a cycle $C$ of length $\ell$ containing $v_1$, but no proper cycle of length $\ell+1$ containing $v_1$ in $(G,c)$,
and suppose there is no monochromatic triangle contains two vertices in $V(C)$ and a vertex in $V(G)\setminus V(C)$.
If the number of follower vertices is larger than 2, then

(1) \ for every follower vertex $w$ and every two distinct vertices $v_i$ and $v_j$ in $C$, we have
$c(wv_i)\neq c(wv_j)$, and

(2) \ $C$ has the $DP_w$ for every $w \in W_2$, where $DP_w$ is defined in Definition \ref{DP}.
\end{claim}

In this paper, we solve the conjecture by adding a looser condition that the edge-colored complete graph
can have monochromatic triangles but not any two joint monochromatic triangles. Our main result is stated as follows.

\begin{theorem}\label{main}
Let $(G,c)$ be an edge-colored complete graph on $n\geq 3$ vertices. If $\delta^c(G)\geq \frac{n+1}{2}$ and $(G,c)$
contains no joint monochromatic triangles, then $(G,c)$ is properly vertex-pancyclic.
\end{theorem}

The following new definitions are needed in the sequel.

\begin{defi}
Let $v_i$ and $v_j$ be two distinct vertices on a cycle $C$.
The distance between $v_i$ and $v_j$ (denoted by $d_{ij}$) is the length of the shortest path
of $v_i\overrightarrow{C}v_j$ and $v_i\overleftarrow{C}v_j$.
Apparently, $d_{ij}=d_{ji}=min\{|i-j|,|j-i|,|i+l-j|,|j+l-i|\}\leq \frac{l}{2}$.
Furthermore, we say that $v_i$ is in front of $v_i$ on $C$ if $d_{ij}=|v_i\overrightarrow{C}v_j|$.
\end{defi}

The authors in \cite{FM} gave a property on set version.

\begin{defi} \noindent {\bf (Set version)}
In an edge-colored complete graph $(G,c)$, a set $A$ of vertices is
said to have dependence property with respect to a vertex $v\notin A$
(denoted by $DP_v$) if $c(aa')\in\{c(va),c(va')\}$ for every two vertices $a,a'\in A$.
\end{defi}

Then, based on the definition on set version,
we give a similar definition on vertex version.

\begin{defi} \label{DP}
\noindent {\bf (Vertex version)}
In an edge-colored complete graph $(G,c)$, a pair $(u,w)$ of distinct vertices is
said to have dependence (independence) property with respect to another vertex $v$
(denoted by $DP_v$) if $c(uw)\in\{c(vu),c(vw)\}$ ($c(uw)\notin\{c(vu),c(vw)\}$)
for every two vertices $u,w\in A$.
The set of these vertices pairs is denoted by $D_v$ ($I_v$).
\end{defi}

The following is an important fact appearing in \cite{FM}, which will be used later.

\begin{fact}\cite{FM}\label{fact1}
If a set $A$ of vertices in an edge-colored complete graph $(G,c)$
has the $DP_v$ for some vertex $v\notin A$, then there exists a vertex $a\in A$ such that

$a)$ \ $d^c_{A}(a)\leq \frac{|A|+1}{2}$, and

$b)$ \ if $|A|\geq 2$, then at least one of the colors used at $a$ in $A$ is $c(va)$.
\end{fact}

\begin{theorem}\cite{FM}\label{C3}
Let $(G,c)$ be an edge-colored completed graph on $n\geq 3$ vertices such that $\delta^c(G)\geq \frac{n+1}{2}$.
Then every vertex of $(G,c)$ is contained in a rainbow triangle.
\end{theorem}

\begin{theorem}\cite{FM}\label{C4}
Let $(G,c)$ be an edge-colored completed graph on $n\geq 3$ vertices such that $\delta^c(G)\geq \frac{n+1}{2}$.
If $n\geq 4$, then every vertex is contained in a proper cycle of length 4,
and if $n\geq 13$, every vertex is contained in a proper cycle of length at least 5.
\end{theorem}

In this paper, a subgraph induced by $V(C)$ union a vertex $w\in V(G)\setminus V(C)$ is denoted by $G_C$.
Then, from Chen, Huang and Yuan \cite{CHY} we can get the following properties about $G_C$.

\begin{pro}\label{pro1}
Suppose there is no proper cycle of length $\ell+1$ containing $v_1$ in $(G,c)$.
Let $P$ be a proper path on $C$.
Let $v_a$ and $v_b$  be two distinct vertices in $V(P)$ such that $c(wv_a)=c(wv_b)$.
If there are no monochromatic triangles containing $w$ in $(G_C,c)$,
then $c(wv_{a- i})=c(wv_{b- i})$ ($c(wv_{a+ i})=c(wv_{b+ i})$) for $1\leq i\leq \ell$
if $w$ follows the colors of $C$ increasingly (decreasingly),
($v_{a\pm i}$ and $v_{b\pm i}$ are in $V(P)$).
\end{pro}

\begin{pro}\label{pro2}
Suppose there is no proper cycle of length $\ell+1$ containing $v_1$ in $(G,c)$.
Let $P$ be a proper path on $C$.
If $w $ follows the colors of $P$ increasingly (decreasingly)
and there are two vertices $v_i,v_j \in V(P)$ such that
$v_j$ is in front of $v_i$ and $c(wv_i)\neq c(wv_j)$,
then we have $(v_{i+1},v_{j+1})\in D_w$
($(v_{i-1},v_{j-1})\in D_w$).
\end{pro}

\begin{pro}\label{pro3}
Suppose there is no proper cycle of length $
ell+1$ containing $v_1$ in $(G,c)$.
Let $P$ be a proper path on $C$.
If $w$ follows the colors of $P$ increasingly (decreasingly)
and there are two vertices $v_i,v_j \in V(P)$ such that
$v_j$ is in front of $v_i$ and $(v_i,v_j)\notin D_w$,
then we have $c(wv_{i-1})=c(wv_{j-1})$
($c(wv_{i+1})=c(wv_{j+1})$).
\end{pro}

\section{Proof of Theorem \ref{main}}

In this section we will use a few lemmas and propositions to prove our main result Theorem \ref{main}.

Let $V_1$ be a subset of $V(G)$ and $w$ a vertex of $G$ not in $V_1$.
We give a vertex-induced subgraph $G[V_1]$ a coloring orientation.
First, we orient the edges whose ends are the vertex pairs in $D_w$, that is, orient $v_iv_j$ by $\overrightarrow{v_iv_j}$ if $c(v_iv_j)=c(wv_i)$,
$v_iv_j$ by $\overrightarrow{v_jv_i}$ if $c(v_iv_j)=c(wv_j)$,
and arbitrarily orient $v_iv_j$ if $c(wv_i)=c(wv_j)$.
Next, orient $v_iv_j$ if $(v_i,v_j)\in I_w$ by two inverse arcs
$\overrightarrow{v_iv_j}$ and $\overleftarrow{v_iv_j}$.
Thus, we get a digraph $D(G[V_1])$ of $(G[V_1],c)$,
and $d_{G[V_1]}^c(v)\leq d_{D(G[V_1])}^-(v)+1$.

In the following, we always assume that $(G,c)$ is an edge-colored complete graph on $n\geq 3$ vertices such that $\delta^c(G)\geq \frac{n+1}{2}$, and does not contain any joint monochromatic triangles.
From Theorems \ref{C3} and \ref{C4}, we know that every vertex $v$ of $(G,c)$ is contained in some proper cycles of lengths 3 and 4.
To prove that $(G,c)$ is properly vertex-pancyclic, it suffices to show that if a vertex is contained in a proper $\ell$-cycle in $(G,c)$ for some $\ell$ with $4\leq \ell\leq n-1$, then it is also contained in a proper $(\ell+1)$-cycle in $(G,c)$.

Suppose that $(G,c)$ has a proper cycle $C=v_1v_2\cdots v_\ell v_1$ of length $\ell$ and let $v_1=v$. Let $W_1(C)$ be the set of vertices in $V(G)\setminus V(C)$ such that for each vertex $w\in W_1$,
the edges in $\partial(w, V(C))$ have just one color in $(G,c)$.
Let $W_2(C)$ be the set of vertices in $V(G)\setminus V(C)$
such that each vertex $w\in W_1$ follows the colors of $C$.
Let $W_3(C)=V(G)\setminus (V(C)\cup W_1(C)\cup W_2(C))$.
Note that $V(C), W_1(C), W_2(C), W_3(C)$ form a partition of $V(G)$.
For convenience, if $v_k\in V(C)$, we regard $v_{k}$ and $v_{k+\ell}$ (or $v_{k-\ell}$) as the same vertex in the sequel.
First, we analyze the coloring structure on $(G[V(C)\cup w],c)$ for $w\in W_3(C)$.

For each vertex $v\in V(G)$, we use $\partial(v)$ to denote the set of edges incident to $v$ in $G$. Moreover, for two disjoint subsets $X,Y\subseteq V(G)$, we use $\partial(X,Y)$ to denote the
set of edges between $X$ and $Y$ in $G$, i.e., $\partial(X,Y)=\{xy\in E(G): x\in X, y\in Y\}$. $\partial(\{x\},Y)$ will be simply written as $\partial(x,Y)$. The colors of $\partial(\{x\},Y)$ will be simply written as $\mathcal{C}(x,Y)$, that is, $\mathcal{C}(x,Y)=\{c(xy)~|~y\in Y\}$.
For a cycle $C=v_1v_2\ldots v_\ell v_1$ and two vertices $v_i,v_j\in V(C)$ with $1\le i\le j \le \ell$, we use $v_i\overrightarrow{C}v_j$ and  $v_i\overleftarrow{C}v_j$ to denote the paths $v_iv_{i+1}\ldots v_j$ and $v_iv_{i-1}\ldots v_1v_{\ell}v_{\ell-1} \ldots v_j$, respectively.

\begin{lemma}\label{lemma1}
Suppose there is no proper cycle of length $\ell+1$ containing $v_1$ in $(G,c)$. Then for any vertex $w$ in $W_3(C)$, we can find three and only three vertices $v_{x(w)}$, $v_{y(w)}$ and $v_{z(w)}$ in $V(C)$
which can divide $V(C)$ into three subsets:
$$\left\{\begin{array}{lll}
P_C^1(w)&=&\{v_{x(w)},v_{{x(w)}+1},\ldots, v_{y(w)}\},\\[2mm]
P_C^2(w)&=&\{v_{{y(w)}+1},v_{{y(w)}+2},\ldots, v_{z(w)}\},\\[2mm]
P_C^3(w)&=&\{v_{{z(w)}+1},v_{{z(w)}+2},\ldots, v_{{x(w)}-1}\},
\end{array}\right.$$
such that\\

(1) \ $\left\{\begin{array}{l}
c(wv_{x(w)})=c(v_{x(w)}v_{{x(w)}+1})~ while~
c(wv_{{x(w)}-1})\neq c(v_{x(w)}v_{{x(w)}-1}),\\
c(wv_{z(w)})=c(v_{z(w)}v_{{z(w)}-1}) ~while~
c(wv_{{z(w)}+1})\neq c(v_{z(w)}v_{{z(w)}+1}),\\
c(wv_{y(w)})=c(wv_{{y(w)}+1})=c(v_{y(w)}v_{{y(w)}+1});
\end{array}\right.$

(2) \ $\left\{\begin{array}{lll}
c(wv_i)= c(v_iv_{i+1})&  &for~v_i\in P_C^1(w),\\
c(wv_i)= c(v_iv_{i-1})&  &for~v_i\in P_C^2(w),\\
c(wv_i)=c(wv_j)&  & for~v_i,v_j\in P_C^3(w).
\end{array}\right.$.

\end{lemma}

\begin{figure}[htbp]
  \centering
 \scalebox{1}{\includegraphics[width=1.4in,height=1.1in]{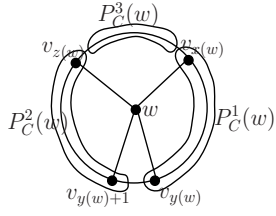}}\\
 \captionsetup{font={scriptsize}}
  \caption{$G[V(C)\cup \{w\}]$ used in the proof of Lemma 3.1.}\label{fig1}
\end{figure}

\begin{proof}
(1) \ Since $w$ does not follow the colors of $C$, suppose, to the contrary, that $c(wv_i)\neq c(v_iv_{i+1})$ for all $v_i$ in $V(C)$.
Thus, $(G_C,c)$ has no monochromatic triangles containing $w$.
Then from Lemma, we know that $w$ follows the colors of $C$ decreasingly
or $w$ is a single color vertex of $C$, a contradiction.
Thus, there exist two requested vertices $v_{x(w)}$ and $v_{z(w)}$.

As we know, $c(wv_{i+1})\in \{c(wv_{i}),~c(v_{i+1}v_{i+2})\}$ if $c(wv_i)= c(v_iv_{i+1})$ and
$c(wv_{i-1})\in \{c(wv_{i}),~c(v_{i-1}v_{i-2})\}$ if $c(wv_i)= c(v_iv_{i-1})$.
Since $w$ dose not follow the colors of $C$,
there exist two vertices $v_{y_1(w)}$ and $v_{y_2(w)}$
such that $c(wv_{y_1(w)})=c(wv_{y_1(w)-1})$
and $c(wv_{y_2(w)})=c(wv_{y_2(w)+1})$.
Thus, $wv_{y_1(w)}v_{y_1(w)-1}$ and $wv_{y_2(w)}v_{y_2(w)+1}$
are monochromatic.
Then, we have $v_{y(w)}=v_{y_2(w)}=v_{{y_1(w)}-1}$.
As $(G,c)$ has no joint monochromatic triangles, $v_{x(w)}$, $v_{y(w)}$ and $v_{z(w)}$ are the only three requested vertices.

(2) \ Since $v_{x(w)}$, $v_{y(w)}$ and $v_{z(w)}$ are the only three vertices satisfying (1), for any vertex $v_i\in P_C^3(w)$,
we have $c(wv_i)\notin\{c(v_iv_{i-1}),c(v_iv_{i+1})\}$.
If there is a vertex $v_j$ such that $c(wv_i)\neq c(wv_j)$,
then there must exist two adjacent vertices $v_k,v_{k+1}\in P_C^3(w)$ on $C$ such that $c(wv_k)\neq c(wv_{k+1})$.
Thus, $v_kwv_{k+1}\overrightarrow{C}v_k$ is a requested cycle; see Figure 1, a contradiction.
\end{proof}

Note that if $C$ is a proper cycle of length $\ell$ but does not contain  $v_1$, then we can also divide $V(C)$ into three parts such that
each part satisfying the results stated in Lemma \ref{lemma1}.
According to Lemma \ref{lemma1},
for any $w\in W_3(C)$, $wv_{y(w)}v_{y(w)+1}$ is a monochromatic triangle. We denote the color of this monochromatic triangle by $c_w$.
Since $(G,c)$ has no joint monochromatic triangles, $|W_3(C)|\leq \frac{\ell}{2}$.

\begin{lemma}\label{lemma2}
Suppose there is no proper cycle of length $\ell+1$ containing $v_1$ in $(G,c)$. If $|W_2(C)|\geq 2$, then

(1) \ for every vertex $w\in W_2(C)$ and every two distinct vertices $v_i$ and $v_j$ in $C$, we have $c(wv_i)\neq c(wv_j)$, and

(2) \ $C$ has the $DP_w$ for every $w \in W_2(C)$.
\end{lemma}

\begin{proof}
First, we can assert that all vertices in $W_2$ following the colors of $C$ in the same direction. Suppose the contrary holds.
Then, from Claim \ref{claim1}
we have $c(w_1v_i)=c(w_1v_{i+2})$ for all indices $i$ with $1\leq i\leq \ell$. Since $(G,c)$ has no joint monochromatic triangles,
there exists at most one index $k_i$ for $w_i\in W_2(C)$
such that $c(w_iv_{k_i})=c(v_{k_i}v_{{k_i}+2})$ with $1\leq k_i\leq \ell$. Hence, $v_\ell w_1v_1v_3w_2v_4\overrightarrow{C}v_\ell$, $v_1w_1v_2v_4w_5v_4\overrightarrow{C}v_1$
and $v_{\ell-2}w_1v_{\ell-1}v_1w_2v_2\overrightarrow{C}v_{\ell-2}$ are cycles of length $\ell+1$ containing $v_1$.
Then, we can easily verify that at least one of them is proper, a contradiction. The assertion follows.
Thus, there are no monochromatic triangles containing a vertex in $W_2(C)$ and two vertices in $V(C)$.
Then, from Claim \ref{claim2} we can get the result.
\end{proof}

In the following, we consider the relation between $W_2(C)$ and $W_3(C)$. According to Lemma \ref{lemma1},
we define some new vertex sets.
Let $R_C(w)=V(v_{y(w)}\overleftarrow{C}v_{r(w)})\subseteq P_C^1(w)$
such that $v_{y(w)+1}v_{y(w)}\ldots v_{r(w)}$
is a longest rainbow subpath of $v_{y(w)+1}v_{y(w)}\ldots v_{x(w)}$.
In a similar way, we define $Q_C(w)=V(v_{y(w)+1}\overrightarrow{C}v_{q(w)})\subseteq P_C^2(w)$.
For convenience, we relabel the vertices of $C$ depending on $w\in W_3(C)$
on a clockwise direction
by $u_1u_2\cdots u_\ell$
such that $u_1=v_{y(w)}$, $u_b=v_{q(w)}$, $u_t=v_{z(w)}$,
$u_{s}=v_{x(w)}$ and $u_a=v_{r(w)}$.

\begin{lemma}\label{lemma3}
Suppose there is no proper cycle of length $\ell+1$ containing $v_1$ in $(G,c)$. If $W_2(C)\neq \emptyset $, then when $|W_3(C)|\geq 2$,
$C$ is an even cycle with two colors appearing alternatively on $C$.
Furthermore, $|W_2(C)|\leq 1$.
\end{lemma}

\begin{proof}
Since $|W_3(C)|\geq 2$, there is a vertex $w\in W_3$ such that $u_1\neq v_1$. Without loss of generality, suppose $w'\in W_2(C)$ follows the colors of $C$ increasingly. To avoid $u_{\ell-1}w'u_\ell wu_2\overrightarrow{C}u_{\ell-1}$ being a requested cycle,
we have $c(wu_\ell)\in \{c(wu_2),c(w'u_\ell)\}$.
If $c(wu_\ell)=c(w'u_\ell)$, then $c(ww')\notin \{c(wu_2),c(w'u_\ell)\}$.
Hence, $u_\ell w'wu_2\overrightarrow{C}u_\ell$ is a requested cycle.
Thus, $c(wu_\ell)=c(wu_2)$. Then, $u_\ell\notin P_C^1(w)$.
Since $w\notin W_1$, we have $|P_C^2(w)|\geq 2$, that is, $u_3\in P_C^2(w)$.
To avoid $u_\ell w'u_2 wu_3\overrightarrow{C}u_\ell$ being a requested cycle, we have $c(w'u_2)=c(w'u_\ell)$. Thus, $C$ is an even cycle with two colors appearing alternatively on $C$. According to Lemma \ref{lemma2},
we know that $|W_2(C)|\leq 1$.
\end{proof}

\begin{pro}\label{p1}
Let $(G,c)$ be an edge-colored complete graph on $n\geq 3$ vertices such that $\delta^c(G)\geq \frac{n+1}{2}$, and not contain any joint monochromatic triangles.

(1) \ If $W_1(C)\neq \emptyset$, then $(G,c)$ is properly vertex-pancyclic.

(2) \ If $|W_2(C)|\geq 2$, then $(G,c)$ is properly vertex-pancyclic.

(3) \ If $|W_2(C)|=1$ and $W_3(C)=\emptyset$, then $(G,c)$ is properly vertex-pancyclic.
\end{pro}

\begin{proof}
Suppose, to the contrary, that there exists a vertex $v$ which is contained in a proper $\ell$-cycle $C$ in $(G,c)$ for some $\ell$ with $4\leq \ell \leq n-1$, but no proper cycle of length $\ell+1$ in $(G,c)$ contains vertex $v$.
\vspace{2mm}

(1) \ Let $w\in W_1(C)$ and $w'\in W_2(C)$.
If $c(ww')\neq c(w,V(C))$, then at least one of $v_1w_2w_1v_3\overrightarrow{C}v_1$,
$v_2w_2w_1v_4\overrightarrow{C}v_2$ and $v_3w_2w_1v_5\overrightarrow{C}v_3$ is a proper
cycle of length $\ell+1$ containing $v_1$.
Thus, $c(ww')= c(w,V(C))$, that is $c(w, W_2(C))=c(w,V(C))$.

If $|W_1(C)|=1$, then for $w\in W_1(C)$, we have
$d^c(w)\leq 1+|W_3|\leq 1+\frac{\ell}{2}=\frac{\ell+2}{2}<\frac{n+1}{2}$, a contradiction.
If $|W_1(C)|\geq 2$, then we can assert has $W_1(C)$ has the $DP_{v_1}$.
Since $(G,c)$ has no joint monochromatic triangles,
there is at most one edge on $C$ colored by $c(w,V(C))$.
Suppose, to the contrary, that $W_1(C)$ has no $DP_{v_1}$.
Then, there are two vertices $w_1,w_2\in W_1$
such that $c(w_1w_2)\notin \{c(v_1w_1),c(v_1w_2)\}$.
Then, at least one of $v_1w_1w_2v_3\overrightarrow{C}v_1$
and $v_1w_2w_1v_3\overrightarrow{C}v_1$
is a requested cycle, a contradiction.
According Fact 1, there exists a vertex $w\in W_1(C)$ such that $d_{W_1(C)}^c(w)\leq \frac{|W_1(C)|+1}{2}$,
and at least one of the colors used in $W_1$ at $w$ is $c(wv_1)$.
Hence, $d^c(w)\leq \frac{|W_1(C)|+1}{2}+|W_3(C)|\leq \frac{|W_1(C)|+1}{2}+\frac{\ell}{2}<\frac{n+1}{2}$,
a contradiction.

\vspace{2mm}

In the following, we might as well suppose $W_1(C)=\emptyset$.

\vspace{2mm}

(2) \ Since $|W_2(C)|\geq 2$, from Lemma \ref{lemma2}, $V(C)$ has the  $DP_w$ for $w\in W_2(C)$. If $W_3(C)= \emptyset$, the result follows.
If $W_3(C)\neq \emptyset$, then from Lemma \ref{lemma3} we have $|W_3(C)|\leq 1$. Thus, there is a vertex $v\in V(C)$
such that $d^c(v)\leq  \frac{|W_2(C)+1|}{2}+1<\frac{n+1}{2}$.

\vspace{2mm}

(3) \ According to the proof of Case 3 of Theorem \ref{th1},
$(G,c)$ has a monochromatic triangle containing $w$.
Otherwise, $V(C)$ has the $DP_w$.
Then, there is a vertex $v\in V(C)$ such that $d^c(v)< \frac{n+1}{2}$.
Let $v_a$ and $v_b$ be two distinct vertices on $C$ such that
$c(v_a,v_b)\in I_w$ (suppose may as well $v_a$ is in front of $v_b$ on $C$). If $wv_iv_{i+d_{a,b}}$ is not monochromatic for any $1\leq i\leq \ell$, then $c(wv_{a-k})=c(wv_{b-k})$ for $k=0,1, \ldots, \ell-1$.
Consequently, $c(wv_{x})= c(wv_{x+kd_{a,b}})$ for every vertex $v_x\in V(C)$ and for every positive integer $k$ (except 1 when $v_x=v_a$).
Furthermore, let $(v_a,v_b)$ be such a vertex pair in $I_w$
that $d_{a,b}= \mbox{min}~ \{d_{i,j}~|~(d_i,d_j)\in I_w\}$,
then $\ell \equiv 0~ (\mbox{mod} ~d_{a,b})$.

Let $wv_xv_y$ be a unique monochromatic triangle containing $w$
(we may as well suppose $v_x$ is in front of $v_y$ on $C$).
Then, we assert $d_{x,y}=d_{a,b}$. Suppose, the contrary holds.
Then, we can get that $c(wv_{i})= c(wv_{i+kd_{a,b}})$
for every vertex $v_i\in V(C)$ and for every positive integer $k$
(except 1 when $v_i=v_a$). Thus, $c(wv_{x})= c(wv_{x+kd_{a,b}})$ for every positive integer $k$. Then, $d_{x,y}=qd_{a,b}$ where $q=2,\cdots,\frac{n-1}{d_{a,b}}$. Otherwise, there is an integer $p$ such that $d_{x+pd_{{a,b},y}}<d_{a,b}$.
Since $c(wv_{x+pd_{a,b}})=c(wv_{x})=c(wv_{y})$,
we have $(v_{x+pd_{a,b}},v_y)\in I_w$,
which contradicts the choice of $(v_a,v_b)$.
Therefore, $\frac{n-1}{d_{a,b}}\geq 3$.
Hence, $d^c(w)\leq 1+d_{a,b}\leq 1+\frac{n-1}{3}<\frac{n+1}{2}$, a contradiction.

For convenience, we relabel the vertices of $C$
by $s_1s_2\cdots s_\ell$ on a clockwise direction
such that $s_1=s_{x}$ and $s_{1+d_{a,b}}=v_y$.
Let $s_q $ be such a vertex that $(s_q,s_{q+d_{a,b}})\in I_w$
while $(s_{q+1},s_{q+d_{a,b}+1})\in D_w$.
Then, from Proposition \ref{pro3}
we have $(s_{i},s_{i+d_{a,b} })\in D_w$ for $q+1<i\leq \ell$
and $c(ws_i)=c(ws_{i+d_{a,b}})$ for $1\leq i<q$.
Since $(G,c)$ has no joint monochromatic,
$\{(s_i,s_{i+d_{a,b}}),1<i<q\}\subseteq I_w$.
Suppose that there is such a vertex pair $(s_f,s_g)$ in $I_w$
while not in $\{(s_i,s_{i+d_{a,b}}), 1<i<q\}$
(we may as well suppose $s_f$ is in front of $s_g$).
Then, according to the above content,
we have $d_{g,f}> d_{a,b}$.
Thus, $c(ws_{i})= c(ws_{i+kd_{g,f}})$
for every vertex $s_i\in V(C)$ and for every positive integer $k$
(except 1 when $s_i=s_f$).
Then, $\mathcal{C}(w,C)=\{c(ws_i)~|~s_i\in V(s_1\overrightarrow{C}s_{1+d_{f,g}})\}$.
Since $c(ws_{1})= c(ws_{1+d_{a,b}})$,
$d^c(w)\leq 1+d_{f,g}-1<\frac{n+1}{2}$.
Thus, $\{(s_i,s_{i+d_{a,b}}),1<i<q\}= I_w$.

Let $V_1=V(s_{q+d_{a,b}}\overrightarrow{C}s_\ell)$ and $V_2=V(C)\setminus V_1$. Since $ V_1\times V_1 \cap I_w=\emptyset$, $V_1$ has the $DP_w$.
Thus, there is a vertex $s\in V_1$
such that $d_{V_1}^c(s)\leq \frac{|V_1|+1}{2}=\frac{n-q-d_{a,b}+1}{2}$.
Since $(s, V_2) \cap I_w=\emptyset$,
$c(ss_i)\in \{c(ws),c(ws_i)\}$ for $s_i\in V_2$.
Thus, $\mathcal{C}(s,V_2)\subseteq \mathcal{C}(w,V_2)\cup \{c(ws)\}$.
Hence, $d^c(s)\leq \frac{n-q-d_{a,b}+1}{2}+d_{a,b}$.
Apparently, $q<d_{a,b}$.
Now we give a coloring orientation for $(G-w,c)$.
Apparently, the edge set which is oriented arbitrarily is $\{s_is_{i+d_{a,b}},2\leq i \leq q\}$.
Thus, $D(G-w)$ has at most $\frac{(n-1)(n-2)}{2}+q-1$ arcs.
Therefore,
\begin{align*}
d^c(G-w) &\leq  d^-(D(G-w))+1 \\
 &\leq \frac{n-2}{2}+\frac{q-1}{n-1}+1\\
 &<\frac{n}{2} +\frac{d_{a,b}-1}{n-1}\\
 &<\frac{n+1}{2}.
\end{align*}
There exists a vertex with color degree less than $\frac{n+1}{2}$.
\end{proof}

Note that there is a special class of vertices $w_i$ in $W_3(C)$
such that $|P_C^1(w_i)|=|P_C^2(w_i)|=1$.
By repeating the proof procedure of Proposition \ref{p1},
we can get Proposition \ref{p2}.

\begin{pro}\label{p2}
Let $(G,c)$ be an edge-colored complete graph on $n\geq 3$ vertices and with no  joint monochromatic triangles such that $\delta^c(G)\geq \frac{n+1}{2}$.
If there exist such vertices $w_i\in W_3(C)$
that $|P_C^1(w_i)|=|P_C^2(w_i)|=1$ ,
then $(G,c)$ is properly vertex-pancyclic.
\end{pro}

Hence, in the following we suppose that
either $|P_C^1(w)|$ or $|P_C^2(w)|$
is larger than 1 for each $w\in W_3$.

\begin{lemma}\label{lemma4}
Suppose there is no proper cycle of length $\ell+1$ containing $v_1$ in $(G,c)$. If $|W_2(C)|\leq 1$, then $|\mathcal{C}(w,C)|\geq 3$ for $w\in W_3(C).$
\end{lemma}

\begin{proof}
The result follows when $|W_3(C)|=1$.
Suppose now $|W_3(C)|\geq 2$.
If $W_2\neq \emptyset$, then from Lemma \ref{lemma3},
we know that $C$ is an even cycle with two colors appearing alternatively on $C$. Let $w_1\in W_2$ and $w_2\in W_3$.
Without loss of generality, assume that $w_1$ follows the colors of $C$ increasingly. Then, we assert $c(w_1w_2)\in \mathcal{C}(w_2,C)$.
If $u_1\neq v_1$, then to avoid $u_lw_1w_2u_2\overrightarrow{C}u_1$
being a requested cycle, we have $c(w_1w_2)\in \{c(w_2u_2),c(w_1u_1)\}\subseteq \mathcal{C}(w_2,C)$.
If $u_1= v_1$, then to avoid $w_1u_4u_3u_5\overrightarrow{C}u_1 w_2w_1$
being a requested cycle, we have $c(w_1w_2)\in \{c(w_2u_4),c(w_1u_1)\}\subseteq \mathcal{C}(w_2,C)$.
Thus, $d^c(w)\leq |\mathcal{C}(w,C)|+|W_3(C)|-1$.
Then, $|\mathcal{C}(w,C)|\geq 3$.
\end{proof}

In the following we prove some lemmas to make the coloring structure of $(G[V(C)\cup\{w\}],c)$ clear for $w\in W_3(C)$.

\begin{lemma}\label{lemma5}
Suppose there is no proper cycle of length $\ell+1$ containing $v_1$ in $(G,c)$. For $w\in W_3(C)$, if $|P_C^3(w)|\geq 2$,
then $R_C(w) =P_C^1(w)$ and $Q_C(w) =P_C^2(w)$.
\end{lemma}

\begin{proof}
We prove Lemma \ref{lemma5} by contradiction.
Assume that there is a vertex $u_i\in R_C(w)$ such that $c(wu_s)=c(wu_i)$. Then, $(u_s,u_i)\in I_w$.
From Proposition \ref{pro3},
we have $c(wu_{s-1})=c(wu_{i-1})$.
Since $|P_C^3(w)|\geq 2$,
$c(wu_{{s}-2})=c(wu_{i-1})\neq c({w}u_{{i}-2})$.
Thus, $c(u_{s-1}u_{i-1})\in \{c(u_{{s}}u_{{s}-1}),c(u_{i-1}u_i)\}$;
otherwise,  $wu_{i-2}\overleftarrow{C}u_{{s}-1}u_{i-1}\overrightarrow{C}u_{{s}-2}w$
is a requested cycle.
Then, $c(u_{{s}-1}u_{i-1})\neq c(u_{{s}-2}u_{{s}-1})$.
Therefore, $wu_{i-2}\overleftarrow{C}u_{s} u_{i}\overrightarrow{C}u_{{s}-1}u_{i-1}w$
is a requested cycle, a contradiction.
\end{proof}

Now we define a cycle $C_{u_i}=u_{i-1}wu_{i+1}\overrightarrow{C}u_{i+1}$ where $u_i\in \{u_1,u_2\}$.
Taking an example of which $C_{u_1}$ is proper,
we can get a conclusion.
Suppose there is no proper cycle of length $\ell+1$ containing $v_1$ in $(G,c)$. If there is a vertex $u_i$ such that $c(u_1u_i)\neq c(u_iu_{i+1})$, then $c(u_1u_{i-1})\in\{c(u_1u_i),c(u_{i-1}u_{i-2})\}$.
Otherwise,  $wu_{2}\overrightarrow{C}u_{i-1}u_1u_{i}\overrightarrow{C}u_{\ell}w$
is a requested cycle.
Thus, $c(u_1u_{i-1})\neq c(u_{i}u_{i-1})$.
By repeating this proof procedure,
we can get that $c(u_1u_k)\in \{c(u_1u_{k+1}),c(u_{k-1}u_k)\}$
and $c(u_1u_k)\neq c(u_ku_{k+1})$
for $u_k\in V(u_{i-1}\overleftarrow{C}u_{2})$.
Notice that if $c(u_1u_j)=c(u_1u_i)$,
then $c(u_1u_{j-1})=c(u_1u_i)$,
and once there is a vertex $u_j$ such that
$c(u_1u_j)=c(u_ju_{j-1})$,
then $c(u_1u_k)=c(u_ku_{k-1})$ for $u_k\in V(u_{j}\overleftarrow{C}u_{2})$.
Consequently, $c(u_1u_k)\in \{c(u_1u_{i}), c(u_ku_{k-1})\}$ for $u_k\in V(u_{i-1}\overleftarrow{C}u_{2})$.
In a similar way, if there is a vertex $u_i$ such that $c(u_1u_i)\neq c(u_iu_{i-1})$,
then $c(u_1u_k)\in \{c(u_1u_{i}), c(u_ku_{k+1})\}$ for $u_k\in V(u_{i+1}\overrightarrow{C}u_{\ell})$.

\begin{lemma}\label{lemma6}
Suppose there is no proper cycle of length $\ell+1$ containing $v_1$ in $(G,c)$. For $w\in W_3(C)$, we have

(1) \ $R_C(w)\setminus\{u_a\}$ and $Q_C(w)\setminus \{u_{b}\}$ has the $DP_w$;

(2) \ $(R_C(w)\setminus \{{u_a}\},Q_C(w)\setminus\{u_{b}\})\subseteq D_w$;

(3) \ if neither $P_C^1(w)\setminus R_C(w)$ nor $P_C^2(w)\setminus Q_C(w)$ is empty, and then $(u_a,Q_C(w)\setminus \{u_2\})\cup(R_C(w)\setminus \{u_1\},u_{b})\subseteq D_w$.
\end{lemma}

\begin{proof}
(1) \ Suppose, to the contrary, that there exist two vertices $u_i$ and $u_j$ in $R_C(w)\setminus \{u_a\}$ that have no $DP_w$ (we might as well assume that $u_i$ is in front of $u_j$).
Then, according to Proposition \ref{pro3},
we have $c(wu_{i-1})= c(wu_{j-1})$, a contradiction.

\vspace{2mm}

(2) \ Suppose, to the contrary, that $(u_i,u_j)\in I_w$,
where $u_i\in R_C(w)\setminus \{{u_a}\}$ and $u_j\in Q_C(w)\setminus\{u_b\}$.
Let $C'=wu_2\overrightarrow{C}u_ju_i\overrightarrow{C}u_1u_{i-1}
\overleftarrow{C}u_{j+1}w$ be a cycle of length $\ell+1$ containing $v_1$.
If $u_{i-1}\neq u_a$,
then according to (1)
we have $c(u_1u_{i-1})\in \{c(wu_1),c(wu_{i-1})\}$.
Thus, $c(u_1u_{i-1})\notin \{c(u_1u_2),c(u_{i-1}u_{i-2})\}$.
Since $c(wu_1)\neq c(wu_{j+1})$,
$C'$ is proper, a contradiction.
If $u_{i-1}= u_a$,
then we can get that $c(u_1u_{a})\in \{c(u_1u_l),c(u_{a}u_{a-1})\}$.
Otherwise, $C'$ is proper.
Thus, $c(u_1u_{a})\neq c(v_{a}v_{a+1})$.
Hence, $c(u_1u_k)\in\{c(u_1u_{k+1}),c(u_ku_{k-1})\}$
for $u_k\in V(u_{a-1}\overrightarrow{C}u_{2})$.
Then, $c(u_1u_{j+1})\neq c(u_{j+1}u_{j+2})$.
Therefore, $wu_2\overrightarrow{C}u_ju_i\overrightarrow{C}u_1u_{j+1}
\overrightarrow{C}u_{i-1}w$ or $wu_\ell\overleftarrow{C}u_i u_j\overleftarrow{C}u_1u_{j+1}\overrightarrow{C}u_{i-1}w$
is a requested cycle, a contradiction.

\vspace{2mm}

(3) \ We prove this statement by classified discussion.
If $c(wu_{a-1})=c_w$,
then $c(u_1u_{a-1})\neq c(u_{a-1}u_a)$.
Since $C_{u_1}$ is proper,
$c(u_1u_k)\in \{c(u_1u_{a-1}), c(u_ku_{k-1})\}$ for $u_k\in V(u_{a-2}\overleftarrow{C}u_2)$.
Therefore, $c(u_1u_k)\notin \{c_w,c(u_ku_{k+1})\}$ for $u_k\in P_C^2(w)$.
Assume that there exist two vertices $u_i\in R_C(w)\setminus\{u_1\}$ and $u_j\in Q_C(w)\setminus\{u_2\}$
such that $(u_i,u_j)\in I_w$.
Then, $c(u_iu_j)\notin \{c(u_{i}u_{i+1}),c(u_{j}u_{j-1})\}$.
Since $c(u_1u_{j+1})\notin \{c_w,c(u_{j+1}u_{j+2})\}$
and $c(wu_{i-1})\neq c(wu_{l})$,
$wu_{l}\overleftarrow{C}u_{i}u_{j}\overleftarrow{C}u_1 u_{j+1}\overrightarrow{C}u_{i-1} w$
is a requested cycle, a contradiction.
Symmetrically, the result holds if $c(wu_{b+1})=c_w$.
Now suppose $c_w\notin\{c(wu_{a-1}),c(wu_{b+1})\}$.
By repeating the proof procedure of (2),
we know that $(u_a,Q_C(w)\setminus \{u_b\})\cup(R_C(w)\setminus \{u_a\},u_b)\subseteq D_w$.
Hence, for $u_i,u_j\in R_C(w)\cup Q_C(w)\setminus \{u_1,u_a,u_b\}$
we have $c(wu_i)\neq c(wu_j)$.
We prove $(u_a,u_b)\in D_w$ in the following cases by contradiction.

If $c(wu_{a-1})= c(wu_{\ell})$,
then $c(u_{1}u_{a-1})\in\{c(u_{a-1}u_{a-2}),c(u_{1}u_{\ell})\}$;
otherwise, $wu_{2}\overrightarrow{C}u_{b }u_{a }\overrightarrow{C}\\ u_{1} u_{a-1}\overleftarrow{C}u_{b+1}w$
is a requested cycle.
When $c(u_{1}u_{a-1})=c(u_{a-1}u_{a-2})$,
we have $c(u_{1}u_k)=c(u_{k}u_{k-1})$ for $u_k\in V(u_{a-1}\overleftarrow{C}u_{2})$.
Hence, $c(u_{1}u_{b+1})=c(u_{b }u_{b+1})\notin \{c(u_{1}u_{\ell}),c(u_{b+1}u_{b+2})\}$.
Then,  $wu_{2}\overrightarrow{C}u_{b}u_{a}\overrightarrow{C}u_{1}u_{b+1}\overrightarrow{C}u_{a-1}w$
is a requested cycle, a contradiction.
When $c(u_{1}u_{a-1})=c(u_{1}u_{\ell})$,
we have $c(u_{1}u_{a-1})\notin\{c(u_{a-1}u_{a-2}),c_w\}$.
Then,  $wu_{\ell}\overleftarrow{C}u_{a}u_{b}\overleftarrow{C}u_{1}u_{a-1}\overrightarrow{C}u_{b+1}w$
is a requested cycle, a contradiction.
Symmetrically, the result holds if $c(wu_{b+1})=c(wu_3)$.

The last case is that $c(wu_{a-1})\neq c(wu_{\ell})$.
According to Proposition \ref{pro1},
we know that there is a vertex $u_i\in R_C(w)$ such that $c(wu_s)=c(wu_i)$.
Since $c(wu_1)\notin \{c(wu_{b+2}),c(wu_{i-1})\}$,
$c(u_{1}u_{b+1})\in\{c(u_{b+1}u_{b}),c(u_{1}u_{\ell})\}$;
otherwise, one of $wu_{2}\overrightarrow{C}u_{b+1}u_{1}\overleftarrow{C}u_{b+2}w$
and $wu_{2}\overrightarrow{C}u_{b+1}u_{1}\\ \overleftarrow{C}u_iu_{s}\overrightarrow{C}u_{i-1}w$
is a requested cycle.
At the same time we have $c(u_{1}u_{b+1})\in \{c(u_{b+1}u_{b+2}),c(u_1u_{2})\}$;
otherwise, $wu_{\ell}\overleftarrow{C}u_{a}u_{b}\overleftarrow{C}u_{1} u_{b+1}\overrightarrow{C}u_{a-1} w$
is a requested cycle.
Therefore, $c(u_{1}u_{b+1})=c(u_{1}u_{\ell})=c(u_{b+1}u_{b+2})$.
Thus, we can get $c(u_{1}u_k)=c(u_ku_{k+1})$ for $u_k\in V(u_{b+1}\overrightarrow{C}u_{\ell})$,
Hence,  $c(u_{1}u_{a-1})=c(u_{a-1}u_a)\notin\{c(u_{1}u_{\ell}),c(u_{a-2}u_{a-1})\}$.
Then, $wu_{2} \overrightarrow{C}u_{b}u_{a}\overrightarrow{C}u_{1}u_{a-1}\overleftarrow{C}u_{b+1}w$
is a requested cycle, a contradiction.
So far, we have completed the proof of (3).
\end{proof}

Lemma \ref{lemma6} (1) and (2) claim that
for $u_i,u_j\in R_C(w)\cup Q_C(w)\setminus \{u_1,u_a,u_b\}$,
we have $c(wu_i)\neq c(wu_j)$.
Lemma \ref{lemma6} (3) claims that
if neither $P_C^1(w)\setminus R_C(w)$ nor $P_C^2(w)\setminus Q_C(w)$ is empty, then for $u_i,u_j\in R_C(w)\cup Q_C(w)\setminus \{u_1\}$,
we have $c(wu_i)\neq c(wu_j)$.
Then, from Lemmas \ref{lemma5} and \ref{lemma6}, we can get the following result.

\begin{pro}\label{p3}
Let $(G,c)$ be an edge-colored complete graph on $n\geq 3$ vertices such that $\delta^c(G)\geq \frac{n+1}{2}$, and not contain any joint monochromatic triangles.
For any $w\in W_3(C)$, if neither $P_C^1(w)\setminus R_C(w)$ nor $P_C^2(w)\setminus Q_C(w)$ is empty, then $(G,c)$ is properly vertex-pancyclic.
\end{pro}

\begin{proof}
Suppose, to the contrary, that there exists a vertex $v$ which is contained in a proper $\ell$-cycle $C$
in $(G,c)$ for some $\ell$ with $4\leq \ell\leq n-1$, but no proper cycle of length $\ell+1$ in $(G,c)$ contains vertex $v$.
According to Lemma \ref{lemma5},
we can get $|P_C^3(w)|\leq 1$.
If $P_C^3(w)=\emptyset$,
then $c(wu_{a-1})=c(wu_{b+1})=c_w$;
otherwise, $(G,c)$ has a requested cycle.
Thus, $c(u_2u_{b+2})\neq c(u_{b+2}u_{b+3})$.
Then,
$wu_3\overrightarrow{C}u_{b+1}u_1u_2u_{b+2}\overrightarrow{C}u_\ell w$
or $wu_1u_{b+1}\overleftarrow{C}u_2u_{b+2}\overrightarrow{C}u_\ell w$
is a requested cycle, a contradiction.
If $|P_C^3(w)|= 1$,
then $c(w,P_C^3(w))=c(wu_{a-1})=c(wu_{b+1})=c_w$.
Thus, $c(u_1u_{t+1})=c(u_{t+1}u_s)$.
Since $C_{u_1}$ is proper,
we have $c(u_1u_k)=c(u_ku_{k+1})$ for $u_k\in P_C^1(w)$.
Hence, $wu_1u_{a-1}$ is monochromatic, a contradiction.
\end{proof}

Thus, if there is no proper cycle of length $\ell+1$ containing $v_1$ in $(G,c)$ and $W_3(C)\neq \emptyset$,
then for any $w\in W_3(C)$, either $P_C^1(w)\setminus R_C(w)$ or $P_C^2(w)\setminus Q_C(w)$ is empty.
Hence, we might as well suppose that $P_C^2(w)\setminus Q_C(w)=\emptyset$, that is $Q_C(w)=P_C^2(w)$.

\begin{pro}\label{p4}
Let $(G,c)$ be an edge-colored complete graph on $n\geq 3$ vertices such that $\delta^c(G)\geq \frac{n+1}{2}$, and not contain any joint monochromatic triangles.
If for any $w\in W_3$, $P_C^1(w)=V(u_3\overrightarrow{C}u_2)$ and $P_{C_{u_1}}^1(u_1)=V(u_3\overrightarrow{C_{u_1}}u_2)$, and $u_4\overrightarrow{C}u_2$ is a rainbow path,
then $(G,c)$ is properly vertex-pancyclic.
\end{pro}

\begin{proof}
Suppose, to the contrary, that there exists a vertex $v$ which is contained in a proper $\ell$-cycle $C$
in $(G,c)$ for some $\ell$ with $4\leq \ell \leq n-1$, but no proper cycle of length $\ell+1$ in $(G,c)$ contains vertex $v$.
Since $wu_{k-1}\overleftarrow{C}u_2u_k\overrightarrow{C}u_1w$
is of length $\ell+1$ and contains $v_1$,
we have $c(u_2u_k)\in \{c(u_2u_3),c(u_ku_{k+1})\}$ for $u_k\in V(u_3\overrightarrow{C}u_1)$.
According to Lemma \ref{lemma6},
we have $V(u_5\overrightarrow{C}u_\ell)$ has the $DP_w$.
Then, there is a vertex $u_p$
such that $d^c_{V(u_5\overrightarrow{C}u_\ell)}(u_p)\leq \frac{\ell-3}{2}$.
If $c(u_2u_p)=c(u_pu_{p+1})$,
then $d^c(u_p)\leq \frac{\ell+1}{2}<\frac{n+1}{2}$, a contradiction.
Thus, $c(u_2u_p)=c(u_2u_3)\neq c(u_pu_{p+1})$.
Then, to avoid $wu_2u_p\overrightarrow{C}u_1u_i\overleftarrow{C}u_3u_{i+1}
\overrightarrow{C}u_{p-1}w$ for $u_i\in V(u_5\overrightarrow{C}u_{p-1})$
and $wu_{i-1}\overleftarrow{C}u_pu_2\overleftarrow{C}u_iu_3
\overrightarrow{C}u_{p-1}w$ for $u_i\in V(u_{p+1}\overrightarrow{C}u_{\ell})$
being requested cycles,
we have $c(u_3u_i)\in \{c(wu_3),c(wu_i)\}$ for $u_i\in V(u_5\overrightarrow{C}u_1)\setminus \{u_p\}$.
Note that if there is another vertex $u$ such that $c(uu_2)=c(u_2u_3)\neq c(wu)$, we have $(u_3,V(u_4\overrightarrow{C}u_{p-1}))\subseteq DP_w$.
Then, $u_3\in R_C(w)$; otherwise, $(G,c)$ has joint monochromatic triangles. Thus, there is a vertex in $V(u_3\overrightarrow{C}u_\ell)$
whose color degree is less than $\frac{n+1}{2}$, a contradiction.
Hence, $u_p$ is the unique vertex such that $c(u_2u_p)=c(u_2u_3)\neq c(u_pu_{p+1})$.
If $n\leq 7$, we can easily find a vertex of $d^c(u)<\frac{n+1}{2}$.
If $n>8$, we give $(G[V(u_5\overrightarrow{C}u_\ell)],c)$ a coloring orientation.
If there is a distinct vertex $u$ such that
$d^+_D(u)\geq \frac{\ell-5}{2}$,
then $d^c(u)\leq \frac{\ell+1}{2}<\frac{n+1}{2}$, a contradiction.
Thus, $d^+_D(u_p)\geq \frac{(\ell-4)(\ell-5)}{2}-\frac{(\ell-5)(\ell-6)}{2}=\ell-5$.
Then $d^c(u_p)\leq 4$, a contradiction.
\end{proof}

\begin{lemma}\label{lemma7}
Suppose there is no proper cycle of length $\ell+1$ containing $v_1$ in $(G,c)$. Then for $w\in W_3(C)$ with $P_C^3(w)\neq \emptyset$, we have

(1) \ $(R_C(w)\setminus \{{u_1},u_a\},u_b)\cup (u_a,Q_C(w)\setminus\{u_2,u_b\})\subseteq D_w$.

(2) \ if all of $|P_{C}^1(w)|,~|P_{C}^2(w)|$ and $|P_{C}^3(w)|$ are larger than 1, then $(u_a,u_b)\in D_w$.
\end{lemma}

\begin{proof}
(1) \ The proof method is the same as that of Lemma \ref{lemma6} (2) and (3).

\vspace{2mm}

(2) \ From (1) and Lemma \ref{lemma6} (2) ,
we can get $(R_C(w)\setminus \{u_1\},Q_C(w)\setminus \{u_2\})\setminus \{(u_a,u_b)\}\subseteq D_w$.
If $c(wu_l)=c(wu_3)$,
we have $a=\ell$ and $b=3$.
Thus, $c(w,P_C^3(w))\neq c(wu_3)$;
otherwise, $d^c(w)=2<\frac{n+1}{2}$, a contradiction.
Thus, $c(w,P_C^3(w))\neq c(wu_\ell)$ or $c(w,P_C^3(w))\neq c(wu_3)$ holds. Without loss of generality, suppose $c(w,P_C^3(w))\neq c(wu_3)$.

Suppose, to the contrary, that $c(u_au_b)\notin \{c(u_au_{a+1}),c(u_{b}u_{b-1})\}$.
Since $|P_C^3(w)|\geq 2$, we have
$u_a=u_s$ and $u_b=u_t$.
Since $c(wu_{t+2})\neq c(wu_3)$,
we can get $c(u_2u_{t+1})\in \{c_w,c(u_{t+1}u_t)\}$.
If $c(u_2u_{t+1})=c_w$,
then $c(w,P_C^3(w))\neq c_w$.
To avoid $wu_1\overleftarrow{C}u_su_t\overleftarrow{C}u_2 u_{t+1}\overrightarrow{C}\\u_{s-1}w$
being a requested cycle,
we have $c(u_2u_{t+1})=c(u_{t+1}u_{t+2})$.
Hence, $c(u_2u_{s-1})=c(u_{s}u_{s-1})$.
Then, $wu_3\overrightarrow{C}u_t u_s\overrightarrow{C}u_2u_{s-1}\overleftarrow{C}u_{t+1}w$
or $wu_1\overleftarrow{C}u_su_t\overleftarrow{C}u_2u_{s-1}\overleftarrow{C}u_{t+1}w$
is a requested cycle, a contradiction.
If $c(u_2u_{t+1})=c(u_tu_{t+1})\neq c_w$,
then $c(u_2u_{t+1})\neq c(u_{t+2}u_{t+1})$.
Hence, $wu_3\overrightarrow{C}u_t u_s\overrightarrow{C}u_2u_{t+1}\overleftarrow{C}u_{s-1}w$
is a proper cycle of length $\ell$ containing $v_1$, a contradiction.
\end{proof}

Lemmas \ref{lemma6} (2) and \ref{lemma7} claim that
if all of $|P_{C}^1(w)|,~|P_{C}^2(w)|$ and $|P_{C}^3(w)|$ are
larger than 1,
then for $u_i,u_j\in R_C(w)\cup Q_C(w)\setminus \{u_1\}$,
we have $c(wu_i)\neq c(wu_j)$.

\begin{lemma}\label{lemma8}
Suppose there is no proper cycle of length $\ell+1$ containing $v_1$ in $(G,c)$.
If $|W_3(C)|\geq 2$, then for each vertex $w\in W_3(C)$ such that
neither $wv_1v_l$ nor $wv_1v_2$ is monochromatic, we have that
both $C_{u_1}$ and $C_{u_2}$
are proper cycles of length $\ell$ containing $v_1$.
\end{lemma}

\begin{proof}
Let $w$ be a vertex such that $wv_1v_i$ is not monochromatic for $i=2,\ell$, and $w'$ be a distinct vertex.
We again relabel the vertices of $C$ depending on $w'\in W_3(C)$ by $u'_1u'_2\cdots u'_\ell$
in a clockwise direction such that $u_1=v_{y(w')}$,
$u_b=v_{q(w')}$, $u_t=v_{z(w')}$,
$u_{s}=v_{x(w')}$ and $u_a=v_{r(w')}$.
Without loss of generality, suppose $|P_C^1(w)|\geq 2$.
Then, assume, to the contrary, that $C_{u_2}$
is not proper, that is, $|P_C^2(w)|=1$ and $c(w,P_C^3(w))=c_w$
if $P_C^3(w)\neq \emptyset$.
Then, $u_1,w'\in W_3(C_{u_1})$ and
$C_{u_1}$ is a proper cycle containing $v_1$.

In the first case we assume that $C_{u_2'}$ is proper.
Then, the coloring of $\partial(u_2',C_{u_2'})$ follows the statement in Lemma \ref{lemma1}. If $c(wu_2')=c_w$,
then from Lemma \ref{lemma4} we have $|R_{C_{u_2'}}(w)|\geq 4$.
Apparently, $u_1\notin P_{C_{u_1'}}^1(u_2')$,
and then $u_1,u_l\in P_{C_{u_1'}}^2(u_2')\cup P_{C_{u_1'}}^3(u_2')$.
Thus, $u_{\ell-1}wu_2'u_1\overrightarrow{C}u_{\ell-1}$
or $u_{\ell-2}wu_2'u_\ell\overrightarrow{C}u_{\ell-2}$
is a requested cycle, a contradiction.
If $c(wu_2')\neq c_w$, then $u_2'\in P_{C}^1(w)$ and $c(wu_2')=c(u_2'u_3')$.
When $c(wu_2')=c(wu_\ell)$, we have $w\in P_{C_{u_1}}^2(u_2')$.
Thus, $c(u_3'u_4')=c_w$.
Then, $u_4'u_2'u_3'wu_1'\overleftarrow{C}u_4'$ is a requested cycle, a contradiction.
When $c(wu_2')\neq c(wu_\ell)$, we have $w\in P_{C_{u_1}}^3(u_2')$.
Note that $c(wu_\ell')\neq c(wu_2')$; otherwise, $wu_2'u_\ell'$ is monochromatic. Since $u_\ell'wu_2'u_1'\overrightarrow{C}u_\ell'$ is a proper cycle of length $\ell+1$, we have $u_1'=v_1$.
Then, $u_1'ww'u_2'u_4'\overrightarrow{C}u_1'$ or $u_1'wu_3'u_2'u_4'\overrightarrow{C}u_1'$
is a requested cycle, a contradiction.

In the second case we assume that $C_{u_1'}$ is proper.
Then, the coloring of $\partial(u_1',C_{u_1'})$ follows the statement in Lemma \ref{lemma1}.
If $c(wu_1')=c_w$, then
it is easy to verify that there
is a requested cycle in $(G,c)$, a contradiction.
If $c(wu_1')\neq c_w$, then $u_1'\in P_{C}^1(w)$ and $c(wu_1')=c(u_1'u_2')$.
When $c(wu_1')=c(wu_\ell)$, we have $V(w\overleftarrow{C_{u_1}}u_2')\subseteq P_{C_{u_1}}^2(u_1')$.
Since $w'wu_1'u_3'\overrightarrow{C}w'$ is a proper cycle of length $\ell$, we have $u_2'=v_1$.
To avoid $w'wu_2'u_1'u_4'\overrightarrow{C}w'$ being a requested cycle,
we have $c(u_1'u_4')=c(u_1'u_2')$.
Then, $|\mathcal{C}(w,C)|<4$, a contradiction.
When $c(wu_1') \neq c(wu_\ell)$,
furthermore if $c(ww')=c(u_1'u_2')$, then $ww'u_1'$ is monochromatic.
Thus, $c(ww')=c_w$, it is easy to verify that there
is a requested cycle in $(G,c)$, a contradiction.
\end{proof}

In the following we prove an important lemma which can
transform a cycle $C$ at $w\in W_3(C)$ with $|P_C^3(w)|\geq 3$
into a new cycle $C_{u_i}$ at $u_i\in W_3(C_{u_i})$ with $|P_{C_{u_i}}^3(u_i)|\leq 1$, $i=1,2$
under the condition that there is no proper cycle of length $\ell+1$ containing $v_1$ in $(G,c)$.

\begin{lemma}\label{lemma9}
Suppose there is no proper cycle of length $\ell+1$ containing $v_1$ in $(G,c)$. Then for $w\in W_3(C)$, if $|P_C^3(w)|\geq 3$
then $u_1\in W_3(C_{u_1})$ with $|P_{C_{u_1}}^3(u_1)|\leq 1$
or $u_2\in W_3(C_{u_2})$ with $|P_{C_{u_2}}^3(u_2)|\leq 1$.
\end{lemma}

\begin{proof}
In the first case we assume that both $|P_C^1(w)|$ and $|P_C^2(w)|$ are larger than 1.
Then, according to Lemmas \ref{lemma6} (2) and \ref{lemma7},
we have $c(wu_l)\neq c(wu_3)$.
Apparently, $c_w\notin\{c(wu_\ell,c(wu_3))\}$, and then $C_{u_1}$ and $C_{u_2}$ are proper.
Since $|P_C^3(w)|\geq 3$, there is a vertex $u_p\in P_C^3(w)$
such that $u_{p-1},u_{p+1}\in P_C^3(w)$; see Figure \ref{fig2}.

\begin{figure}[htbp]
  \centering
 \scalebox{1}{\includegraphics[width=1.2in,height=1.3in]{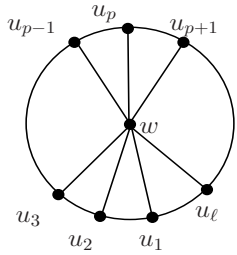}}\\
 \captionsetup{font={scriptsize}}
  \caption{$G[V(C)\cup \{w\}]$ used in the proof of Lemma 3.9.}\label{fig2}
\end{figure}

If $c(w,P_C^3(w))=c_w$,
then $c(u_iu_k)\neq c_w$ for $u_k\in P_C^3(w) $ and $i=1,2$.
It is simple to verify $c(u_1u_p)=c(u_{p+1}u_p)$
and $c(u_2u_p)=c(u_{p-1}u_p)$.
Thus, $c(u_2u_p)\notin \{c(u_pu_{p+1}),c_w\}$ and
$c(u_1u_p)\notin \{c(u_pu_{p-1}),c_w\}$
Then, to avoid $wu_3\overrightarrow{C}u_{p-1} u_1u_2u_{p}\overrightarrow{C}u_{\ell}w$
and $wu_3\overrightarrow{C}u_{p} u_1u_2u_{p+1}\overrightarrow{C}u_{\ell}w$
being requested cycles,
we have $c(u_1u_{p-1})=c(u_{p-1}u_{p-2})$
and $c(u_2u_{p+1})=c(u_{p+1}u_{p+2})$.
Hence, $u_1\in W_3(C_{u_1})$ with $P_{C_{u_1}}^3(u_1)=\emptyset $
and $u_2\in W_3(C_{u_2})$ with $P_{C_{u_2}}^3(u_2)=\emptyset$.

If $c(w,P_C^3(w))=c(wu_3)$,
then $c(w,P_C^3(w))\neq c(wu_\ell )$.
It is simple to verify $c(u_1u_p)\in\{c(u_{p-1}u_p),c(u_1u_\ell)\}\cap \{c(u_{p+1}u_p),c_w\}$
and $c(u_2u_p)\in \{c(u_{p+1}u_p),c(u_2u_3)\}$.
If $c(u_1u_{p})= c(u_{p}u_{p-1})=c_w$,
then $c(u_2u_p)\neq c_w$.
Since $C_{u_1}$ is proper, we have $c(u_1u_k)=c(u_ku_{k-1})$ for $u_k\in V(u_p\overleftarrow{C}u_2)$.
Then, $c(u_1u_3)=c(u_3u_2)\notin \{c(u_3u_4),c(u_1u_\ell)\}$.
Thus, $wu_2u_p\overleftarrow{C}u_{3}u_1\overleftarrow{C}u_{p+1}w$
is a requested cycle, a contradiction.
Hence, $c(u_1u_p)=c(u_{1}u_\ell)=c(u_{p}u_{p+1})$.
Then, $c(u_1u_p)\notin \{c_w,c(u_pu_{p-1})\}$.
Furthermore, if $c(u_2u_p)=c(u_pu_{p+1})$,
then since $C_{u_2}$ is proper,
we have $c(u_2u_k)=c(u_ku_{k+1})$ for $u_k\in V(u_p\overrightarrow{C}u_\ell)$.
Then, $c(u_2u_\ell)=c(u_\ell u_1)\notin\{c(u_2u_3),c(u_\ell u_{\ell-1})\}$.
Thus, $wu_1u_p\overleftarrow{C}u_2u_\ell\overleftarrow{C}u_{p+1}w$
is a requested cycle.
Hence, $c(u_2u_p)=c(u_2u_3)\neq c(u_pu_{p+1})$.
Then, to avoid $wu_3\overrightarrow{C}u_{p-1} u_1u_2u_{p}\overrightarrow{C}u_{\ell}w$
being a requested cycle,
we have $c(u_1u_{p-1})\in\{c_w,\\c(u_{p-1}u_{p-2})\}$.
Since $c(wu_p)\neq c_w$,
it is simple to verify  $c(u_1u_{p-1})\in \{c(u_{p-1}u_{p-2}),c(u_\ell u_{\ell-1})\}$.
Thus, $c(u_1u_{p-1})=c(u_{p-1}u_{p-2})$.
Therefore, $u_1\in W_3(C_{u_1})$ with $P_{C_{u_1}}^3(u_1)=\emptyset $.
Symmetrically, if $c(w,P_C^3(w))=c(wu_\ell)$,
we can get $u_2\in W_3(C_{u_2})$ with $P_{C_{u_2}}^3(u_2)=\emptyset$.

If $c(w,P_C^3(w))\notin \{c_w,c(wu_\ell),c(wu_3)\}$.
then it is simple to verify $c(u_1u_p)\in\{c(u_{p-1}u_p),\\ c(u_1u_\ell)\}\cap \{c_w,c(u_{p+1}u_p)\}$
and $c(u_2u_p)\in\{c(u_{p-1}u_p),c_w)\}\cap \{c(u_{p+1}u_p),c(u_2u_3)\}$.
If $c(u_1u_p)=c(u_{p-1}u_p)=c_w$, then $c(u_2u_p)=c_w$.
Hence, $(G,c)$ has joint monochromatic triangles, a contradiction.
Thus, $c(u_1u_p)=c(u_1u_\ell)=c(u_{p+1}u_p)$
and $c(u_2u_p)=c(u_{p-1}u_p)=c(u_2u_3)$.
Then, $c(u_1u_p)\notin\{c(u_{p-1}u_p), c_w\}$.
Thus, we have $c(u_2u_{p+1})\in\{c_w,c(u_{p+1}u_{p+2})\}$;
otherwise, $wu_3\overrightarrow{C}u_p u_1u_2u_{p+1}\overrightarrow{C}u_{\ell}w$ is a requested cycle.
Since $c(wu_p)\neq c_w$,
it is simple to verify  $c(u_2u_{p+1})\in \{c(u_{p+1}u_{p+2}),c(u_2u_{3})\}$.
Thus, $c(u_2u_{p+1})=c(u_{p+1}u_{p+2})$.
Symmetrically,
we have $c(u_1u_{p-1})=c(u_{p-1}u_{p-2})$.
Since both $C_{u_1}$ and $C_{u_2}$ are proper,
we have $u_1\in W_3(C_{u_1})$ with $P_{C_{u_1}}^3(u_1)=\emptyset$
and $u_2\in W_3(C_{u_2})$ with $P_{C_{u_2}}^3(u_2)=\emptyset$.

Thus the cycle $C$ with a vertex $w\in W_3(C)$ which is of $|P_C^2(w)|\geq 2$ and $|P_C^3(w)|\geq 3$ can
be changed into another cycle $C_{u_i}$ with
$u_i\in W_3(C_{u_i})$ and $P_{C_{u_i}}^3(u_i)=\emptyset$, $i=1$ or $2$; see Figures \ref{fig2} and \ref{fig3}.

\begin{figure}[htbp]
  \centering
 \scalebox{1}{\includegraphics[width=4.5in,height=1.3in]{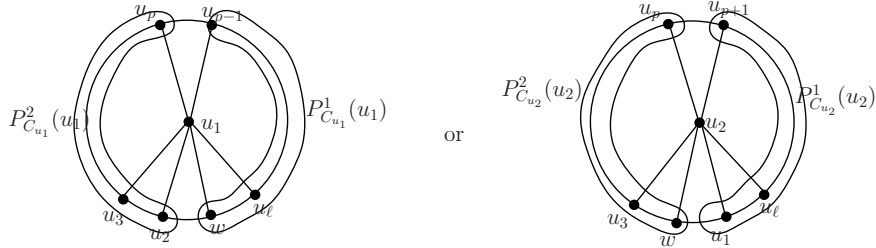}}\\
 \captionsetup{font={scriptsize}}
  \caption{$G[V(C_{u_i})\cup \{u_i\}]$ used in the proof of Lemma 3.9.}\label{fig3}
\end{figure}

In the second case we assume that either $|P_C^1(w)|$ or $|P_C^2(w)|$ is 1.
Without loss of generality, suppose $|P_C^2(w)|=1$.
Then, $C_{u_1}$ is proper.
Since $|P_C^3(w)|\geq 3$, $u_3,u_4,u_5\in P_C^3(w)$.
If $c(wu_3)\notin\{c_w,c(wu_\ell)\}$,
we can get $u_1\in W_3(C_{u_1})$ with $P_{C_{u_1}}^3(u_1)=\emptyset$.
If $c(wu_3)=c_w$,
then it is simple to verify $c(u_1u_4)=c(u_4u_5)$.
Thus, $u_1\in W_3(u_1)$ with $P_{C_{u_1}}^3\subseteq \{u_3\}$.
If $c(wu_3)=c(wu_\ell)$, then $C_{u_2}$ is proper.
To avoid $wu_3u_2u_4\overleftarrow{C}u_1w$
being a requested cycle,
we have $c(u_2u_4)\in\{c(u_{3}u_2),c(u_4u_5)\}$.
Once $c(u_2u_4)=c(u_4u_5)$,
we get $u_2\in W_3(C_{u_2})$ with $P_{C_{u_2}}^3\subseteq \{u_3\}$.
Once $c(u_2u_4)=c(u_{3}u_2)\neq c(u_4u_5)$,
then $u_3\in P_{C_{u_2}}^3(u_2)$.
Thus, $u_\ell\in P_{C_{u_2}}^1(u_2)$.
It is easy to verify $c(u_2u_3)\notin\{c_w,c(u_2u_\ell)\}$.
Then, we can get $u_1\in W_3(C_{u_1})$ with $P_{C_{u_1}}^3(u_1)=\emptyset $.

Thus the cycle $C$ with a vertex $w\in W_3(C)$ which is of $|P_C^2(w)|=1$ and $|P_C^3(w)|\geq 3$ can
be changed into another cycle $C_{u_i}$ with
$u_i\in W_3(C_{u_i})$ and $|P_{C_{u_i}}^3(u_i)|\leq 1$, $i=1$ or $2$;
see Figures \ref{fig4} and \ref{fig5}.

\end{proof}
\begin{figure}[htbp]
  \centering
 \scalebox{1}{\includegraphics[width=1.2in,height=1.1in]{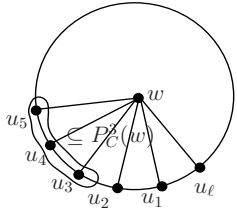}}\\
 \captionsetup{font={scriptsize}}
  \caption{$G[V(C)\cup \{u\}]$ used in the proof of Lemma 3.9.}\label{fig4}
\end{figure}

\begin{figure}[htbp]
  \centering
 \scalebox{1}{\includegraphics[width=4.0in,height=1.4in]{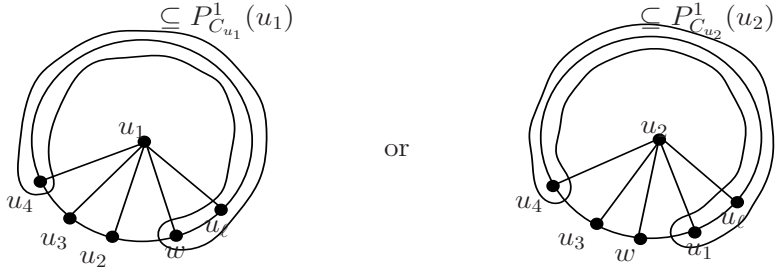}}\\
 \captionsetup{font={scriptsize}}
  \caption{$G[V(C_{u_i})\cup \{u_i\}]$ used in the proof of Lemma 3.9.}\label{fig5}
\end{figure}

From Lemmas \ref{lemma8} and \ref{lemma9},
we get the following important corollary.

\begin{coro}\label{coro1}
Suppose there is no proper cycle of length $\ell+1$ containing $v_1$ in $(G,c)$. Then, $|W_3(C)|\leq 3$.
\end{coro}

\begin{proof}
First, we prove a claim: For a vertex $w$ in $W_3(C)$ such that
$wv_1v_i$ is not monochromatic for $i=2,\ell$,
we have $u_1\notin P_C^1(w')$ and $u_2\notin P_C^2(w')$ for any distinct $w'\in W_3(C)$.
According to Lemma \ref{lemma8}, we know that $C_{u_1}$ and $C_{u_2}$
are proper cycles of length $\ell$ containing $v_1$.
If $u_1\in P_C^1(w')$, then $u_3\in P_C^1(w')$.
Thus, $c(ww')=c(wu_3)$.
According to Lemma \ref{lemma8} again,
we know that $c(u_1u_3)\notin\{c(wu_1),c(u_3u_4)\}$.
Thus, $ww'u_1u_3\overrightarrow{C}w$
is proper cycle of length $\ell+1$ containing $v_1$, a contradiction.
Thus, $u_1\notin P_C^1(w')$.
In a similar way, we can get $u_2\notin P_C^2(w')$.

Suppose, to the contrary, that $|W_3(C)|\geq 4$.
Since $(G,c)$ has no joint monochromatic triangles,
there exists a vertex $w\in W_3(C)$ with $|P_C^3(w)|\geq 4$
such that $wv_1v_i$, $i=2, \ell,$ is not monochromatic; see Figure \ref{fig6}.

\begin{figure}[htbp]
  \centering
 \scalebox{1}{\includegraphics[width=1.2in,height=1.3in]{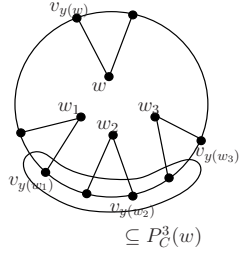}}\\
 \captionsetup{font={scriptsize}}
  \caption{A cycle $C$ with $|W_3(C)|\geq 4$ used in the proof of Corollary 3.1.}\label{fig6}
\end{figure}

According to Lemma \ref{lemma9},
without loss of generality,
suppose $u_1\in W_3(C_{u_1})$ with $|P_{C_{u_1}}^3(u_1)|\leq 1$.
Then, $|W_3(C_{u_1})|\leq 3$.
Since $V(C_{u_1})\cap V(C)=V(u_2\overrightarrow{C}u_\ell)$,
we have $W_3(C_{u_1})=W_3(C)\cup \{u_1\}\setminus \{w\}$,
that is, $|W_3(C_{u_1})|\geq 4$, a contradiction.
\end{proof}

\begin{lemma}\label{lemma10}
Suppose there is no proper cycle of length $\ell+1$ containing $v_1$ in $(G,c)$. Then for $w\in W_3(C)$ with $|P_C^3(w)|\leq 2$,
if both $|P_C^1(w)|$ and $|P_C^2(w)|$ are larger than 1,
there is a vertex set $V_1$ of size $\ell-4$
such that $c(wu)=c(uu_1)=c(uu_2)$ for $u\in V_1$.
\end{lemma}

\begin{proof}
In the first case suppose $c(wu_\ell)=c(wu_3)$.
When $\ell=5$, the result follows apparently.
In the following assume $\ell\geq 6$.
From Lemmas \ref{lemma4} and \ref{lemma7},
we get $P^3_C(w)=\emptyset$.
Let $C_1=wu_\ell u_3\overleftarrow{C}u_1u_4\overrightarrow{C}u_{\ell-1}w$,
$C_2=wu_2u_3u_\ell u_1u_4\overrightarrow{C}u_{\ell-1}w$,
$C_3=wu_3u_\ell\overleftarrow{C}u_2u_4\overrightarrow{C}u_{\ell-1}w$
and $C_4wu_1u_\ell u_3u_2u_4\overrightarrow{C}u_{\ell-1}w$
be cycles of length $\ell+1$ containing $v_1$.
We might as well suppose $u_b=u_3$.
(Note that when $R_C(w)\subsetneq P_C^1(w)$,
we have $u_a=u_\ell$; otherwise, $C_1$ is proper.)
Then, $|R_C(w)|\geq 3$.
To avoid $C_1$ and $C_2$ being proper,
we have $c(u_1u_4)\in \{c_w,c(u_4u_5)\}\cap \{c(u_1u_{\ell-1}),c(u_4u_5)\}$.
Then, $c(u_1u_4)=c(u_4u_5)$.
Since $C_{u_1}$ is proper, we have $c(u_1u_k)=c(u_ku_{k+1})$ for $u_k\in V(u_4\overrightarrow{C}u_\ell)$.
Then, $c(wu_k)\neq c(wu_1)$ for $u_k\in V(u_4\overrightarrow{C}u_\ell)$.
To avoid $C_3$ being proper,
we have $c(u_2u_4)\in \{c_w,c(u_4u_5)\}$.
Thus, $c(u_2u_4)=c(u_4u_5)$;
otherwise, $C_4$ is proper.
Since $C_{u_2}$ is proper,
we have $c(u_2u_k)=c(u_ku_{k+1})$ for $u_k\in V(u_4\overrightarrow{C}u_\ell)$.
The result follows as $V_1=V(u_4\overrightarrow{C}u_\ell)$.

In the second case suppose $c(wu_\ell)\neq c(wu_3)$.
First, we prove a claim:
There is a vertex set $V_1$ of size $\ell-4$
such that
$c(u_iu_2)\neq c_w$ for $u_i\in P_C^1(w)\cap V_1$
and $c(u_ju_1)\neq c_w$ for $u_j\in P_C^2(w)\cap V_1$.
Apparently, the claim holds for $u\in R_C(w)\cup Q_C(w)\setminus \{u_a,u_b,u_1,u_2\}$.
Let $V_0=R_C(w)\cup Q_C(w)\setminus \{u_1,u_2\}$.
If $R_C(w)=P_C^1(w)$,
then the claim follows apparently
when $P_C^3(w)=\emptyset$ or $c(w,P_C^3(w))\neq c(wu_1)$.
While when $|P_C^3(w)|=1$ and $c(w,P_C^3(w))= c(wu_1)$,
suppose, to the contrary, that
$c(u_su_2)=c(u_tu_1)=c_w$.
Then, $c(u_2u_{t+1})\neq c_w$.
Thus, $wu_{t+1}u_2u_s\overrightarrow{C}u_1u_t\overleftarrow{C}u_3w$ is
a requested cycle, a contradiction.
Hence, at least one of $c(u_2u_s)$ and $c(u_1u_t)$ is not $c_w$.
Therefore, the claim follows as $V_1=V_0\setminus \{u_s\}$ or $V_1=V_0\setminus \{u_t\}$.
When $|P_C^3(w)|=2$ and $c(w,P_C^3(w))= c_w$,
since $c(wu_3)\neq c(wu_{t+2})$,
we have $c(u_2u_{t+1})\in \{c_2,c(u_tu_{t+1})\}$.
Thus, $c(u_2u_{t+1})\neq c(u_{t+1}u_{t+2})$.
Then, $c(u_tu_1)\neq c_w$;
otherwise, $wu_2u_{t+1}\overrightarrow{C}u_1u_t\overrightarrow{C}u_3w$
is a requested cycle.
Therefore, the claim follows as $V_1=V_0$.
If $R_C(w)\varsubsetneq P_C^1(w)$,
according to Lemma \ref{lemma5} we know $|P_C^3(w)|\leq 1$.
When $c(wu_{a-1})\neq c(wu_1)$, the claim follows as
$V_1=V_0\setminus \{u_s,u_{s-1}\}$.
When $c(wu_{a-1})= c(wu_1)$, we can get an assertion
that to avoid $wu_2u_{a-1}\overrightarrow{C}u_1u_{a-2}\overleftarrow{C}u_3w$
and $wu_l\overleftarrow{C}u_{a-1}u_2u_1u_{a-2}\overleftarrow{C}u_3w$
being requested cycles,
we have $c(u_1u_{a-2})\in \{c_w,c(u_{a-2}u_{a-3})\}\cap \{c(u_1u_l),c(u_{a-2}u_{a-3})\}$.
In the following we prove the assertion that $|P_C^1(w)\cup P_C^3(w)\setminus R_C(w)|\leq 1$.
If not, since $c(wu_{a-2})=c(wu_\ell)$,
we have $c(u_1u_\ell)\neq c(u_{a-2}u_{a-3})$.
Thus, $c(u_1u_{a-2})=c(u_{a-2}u_{a-3})$.
Since $C_{u_1}$ is proper,
we have $c(u_1u_k)=c(u_ku_{k-1})$ for $u_k\in V(u_{a-2}\overleftarrow{C}u_2)$.
Then, $c(u_1u_3)=c(u_2u_3)\notin \{c(u_3u_4),c(u_1u_\ell)\}$.
Thus,  $wu_\ell\overleftarrow{C}u_{a-1}u_2u_1u_{3}\overrightarrow{C}u_{a-2}w$
is a requested cycle, a contradiction.
The assertion thus follows.
Since $R_C(w)\varsubsetneq P_C^1(w)$, we have $P_C^3(w)=\emptyset$.
Thus, $c(wu_s)=c_w$.
Then, $c(u_1u_t)\neq c_w$;
otherwise, $wu_\ell\overleftarrow{C}u_su_2u_1u_t\overleftarrow{C}u_3w$
is a requested cycle.
Hence, the claim follows as $V_1=V_0\setminus \{u_s,u_t\}$.

Next suppose, to the contrary, that there exists a vertex $u_i\in P_C^1{w}\cap V_1$ such that $c(u_iu_2)\neq c(wu_i)$.
Since $c(u_iu_2)\neq c_w$, to avoid $wu_2u_{a-1}\overrightarrow{C}u_1u_{a-2}\overleftarrow{C}u_3w$
and $wu_\ell\overleftarrow{C}u_{a-1}\\u_2u_1u_{a-2}\overleftarrow{C}u_3w$
being requested cycles,
we have $c(u_1u_{i-1})=c(u_{i-1}u_{i-2})$.
Since $C_{u_1}$ is proper,
we have $c(u_1u_k)=c(u_ku_{k-1})$ for $u_k\in V(u_{i-1}\overleftarrow{C}u_2)$.
In a similar way,
we can get $c(u_1 u_j)=c(wu_j)$ for $u_j\in P_C^2(w)\cap V_1$.
Then, $u_2u_1u_3\overrightarrow{C}u_lu_2$ is a proper cycle of length $\ell$ containing $v_1$.
Then, by repeating this as above,
we can get that $c(u_iu_1)=c(wu_i)$ for $u_i\in P_C^1(w)\cap V_1$
and $c(u_2 u_j)=c(wu_j)$ for $u_j\in P_C^2(w)\cap V_1$.
The result thus follows.
\end{proof}

\begin{lemma}\label{lemma11}
Suppose there is no proper cycle of length $\ell+1$ containing $v_1$ in $(G,c)$.
For $w\in W_3(C)$ with $|P_C^2(w)|=1$ and $|P_C^3(w)|\geq 3$,
if both $c(w,P_C^3(w))=c_w$ and $c(u_1u_3)\neq c(u_2u_3)$,
then $c(u_2u_k)=c(u_3u_k)=c(u_ku_{k+1})$ for $u_k\in V(u_5\overrightarrow{C}u_\ell)$.
\end{lemma}

\begin{proof}
Since $c(wu_3)\neq c(wu_\ell)$, we have
$c(u_1u_4)=c(u_4u_5)$.
From \ref{lemma4} we know that $C_{u_1}$ is proper.
Thus, $c(u_1u_k)=c(u_ku_{k+1})$ for $u_k\in V(u_4\overrightarrow{C}u_\ell)$.
Hence, $c(u_1u_k)\notin \{c(u_ku_{k-1}),c_w\}$ for $u_k\in V(u_4\overrightarrow{C}u_\ell)$.
Since for $u_k\in V(u_5\overrightarrow{C}u_\ell)$,
$wu_1u_{k-1}\overleftarrow{C}u_2u_k\overrightarrow{C}u_\ell w,$
$wu_3\overrightarrow{C}u_{k-1}u_1u_2u_k\overrightarrow{C}u_\ell w$
and $wu_4\overrightarrow{C}u_{k-1}u_1\overrightarrow{C}u_3u_k
\overrightarrow{C}u_\ell w$
are of length $\ell+1$ and contain $v_1$,
we have $c(u_2u_k)\in \{c(u_2u_3),c(u_ku_{k+1})\}\cap \{c_w,c(u_ku_{k+1})\}$
and $c(u_3u_k)\in \{c(u_2u_3),c(u_ku_{k+1})\}$.
Thus, apparently we have $c(u_2u_k)=c(u_ku_{k+1})$ for $u_k\in V(u_5\overrightarrow{C}u_\ell)$.
While if $c(u_3u_k)=c(u_2u_3)\neq c(u_ku_{k+1})$,
then $c(u_1u_3)\notin\{c_w,c(u_3u_{k})\}$.
Thus, $wu_\ell\overleftarrow{C}u_ku_3u_1u_2u_{k-1}\overleftarrow{C}u_4w$
is a requested cycle, a contradiction.
The result thus follows.
\end{proof}

\begin{lemma}\label{lemma12}
Suppose there is no proper cycle of length $\ell+1$ containing $v_1$ in $(G,c)$. Then for $w\in W_3(C)$ with $|P_C^2(w)|=1$ and $|P_C^3(w)|=2$,
we have $c(u_1u_k)=c(u_2u_k)=c(u_ku_{k+1})$ for $u_k\in V(u_5\overrightarrow{C}u_\ell)$.
\end{lemma}

\begin{proof}
In the first case suppose $c(w,P_C^3(w))= c_w$. Then,
to avoid $wu_{3}\overleftarrow{C}u_1u_{4}\overrightarrow{C}u_{\ell}w$
being a requested cycle,
we have $c(u_1u_{4})=c(u_{4}u_{5})$.
Thus, $c(u_{1}u_k)=c(u_ku_{k+1})$ for $u_k\in V(u_4\overrightarrow{C}u_1)$.
Apparently, $c(u_{2}u_k)\in\{c_w,c(u_ku_{k+1})\}$
for $u_k\in V(u_5\overrightarrow{C}u_1)$;
otherwise, $wu_1u_{k-1}\overleftarrow{C}u_2u_k
\overrightarrow{C}u_\ell w$ is a requested cycle.
If $c(u_{2}u_k)=c(u_2u_3)\neq c(u_ku_{k+1})$,
then $c(u_{2}u_k)\neq c(u_1u_{2})$.
Thus, $wu_3\overrightarrow{C}u_{k-1}u_1u_{2}u_{k}\overleftarrow{C}u_{\ell}w$
is a requested cycle, a contradiction.
Hence, $c(u_{2}u_k)=c(u_ku_{k+1})$
for $u_k\in V(u_5\overrightarrow{C}u_1)$.

In the second case suppose $c(w,P_C^3(w))\neq c_w$. Then,
to avoid $wu_4\overrightarrow{C}u_1u_3u_2w$
being a requested cycle,
we have $c(u_1u_{3})\in\{c(u_1u_{\ell}),c(u_{2}u_{3})\}$.
Hence, $c(u_1u_{3})\neq c_w$.
If $c(w,P_C^3(w))\neq c(wu_\ell)$,
then $c(u_1u_{3})\neq c(wu_{3})$.
Thus, $c(u_{2}u_{4})=c(u_{4}u_{5})$;
otherwise, $wu_{3}u_1u_{2}u_{4}\overrightarrow{C}u_{\ell}w$
is a requested cycle.
Therefore, $c(u_{2}u_k)=c(u_ku_{k+1})$ for $u_k\in V(u_5\overrightarrow{C}u_\ell)$.
If $c(w,P_C^3(w))= c(wu_\ell)$,
then to avoid $wu_3u_\ell\overrightarrow{C}u_2u_4\overrightarrow{C}u_{\ell-1}w$
and $wu_3u_2u_4\overrightarrow{C}u_{1}w$
being requested cycles,
we have $c(u_2u_4)\in \{c_w,c(u_4u_5)\}\cap \{c(u_2u_3),c(u_4u_5)\}$.
Thus, $c(u_2u_4)=c(u_4u_5)$.
Since $C_{u_2}$ is proper,
we have $c(u_2u_k)=c(u_ku_{k+1})$ for $u_k\in V(u_4\overrightarrow{C}u_\ell)$.
Furthermore, if $c(u_{1}u_{3})=c(u_{3}u_{4})$,
then since $C_{u_1}$ is proper,
we can get $c(u_{1}u_k)=c(u_ku_{k+1})$ for $u_k\in V(u_5\overrightarrow{C}u_\ell)$;
if $c(u_{1}u_{3})\neq c(u_{3}u_{4})$,
then $C'=u_1u_3\overrightarrow{C}u_\ell u_2u_1$ is proper such that both $|P_{C'}^1(w)|$ and $|P_{C'}^2(w)|$ are larger than 1
or $|P_{C'}^3(w)|=2$.
Thus, according to Lemma \ref{lemma10} or the proof above,
we can get $c(u_{1}u_k)=c(u_ku_{k+1})$ for $v_k\in V(u_5\overrightarrow{C}u_\ell)$.
\end{proof}

\begin{lemma}\label{lemma13}
Suppose there is no proper cycle of length $\ell+1$ containing $v_1$ in $(G,c)$. Then for $w\in W_3(C)$ with $|P_C^2(w)|=1$ and $|P_C^3(w)|\leq 1$, we have $c(u_1u_k)=c(wu_k)$ and $c(wu_k)\in \{c(u_2u_k),c(u_3u_k)\} $ for $u_k\in V(u_5\overrightarrow{C}u_\ell)$.
\end{lemma}

\begin{proof}
First, we prove the assertion that $|P_C^1(w)\cup P_C^3(w)\setminus R_C(w)|\leq 2$.
Suppose, the contrary holds and let $c(wu_{a-1})=c(wu_p)$, $u_p\in R_C(w)$.
If $c(wu_{a-3})= c(wu_\ell)$,
then $c(wu_{a-1})= c(wu_\ell)$ and $|R_C(w)|=3$.
Apparently, $d^c(w)<\frac{n+1}{2}$ if $|W_2(C)|+|W_3(C)|\leq 3$.
Thus, $|W_2(C)|+|W_3(C)|=4$. From Corollary \ref{coro1} and Proposition \ref{p1}
we know $|W_2(C)|=1$ and $|W_3(C)|=3$.
Then, $c(wu_3)=c(wu_{5})$ while $c(wu_4)\neq c(wu_2)$
or $c(wu_2)=c(wu_{4})$ while $c(wu_3)\neq c(wu_1)$, a contradiction.
If $c(wu_{a-3})\neq c(wu_\ell)$,
then $(u_1,u_{a-2})\in DP_w$.
Furthermore, if $u_p\neq u_1$,
then $c(u_1u_{a-2})\in \{c(u_1u_\ell),c(u_{a-2}u_{a-3})\}$;
otherwise,  $wu_2\overrightarrow{C}u_{a-2}u_1\overleftarrow{C}u_pu_{a-1}
\overrightarrow{C}u_{p-1}w$ is a requested cycle.
Since $c(wu_{a-2})\neq c(wu_\ell)$, we have
$c(u_1u_{a-2})=c(u_{a-2}u_{a-3})=c_w$.
Thus, $|R_C(w)|=2$, a contradiction.
If $u_p=u_1$,
then since $|R_C(w)|\geq 3$,
we have $(u_1,u_{a-2})\in DP_w$.
Since $wu_2\overrightarrow{C}u_{a-2}u_1u_{a-1}
\overrightarrow{C}u_\ell w$ is of length $\ell+1$ and contains $v_1$,
we have $c(u_1u_{a-2})\in  \{c(u_{a-1}u_1),c(u_{a-2}u_{a-3})\}$.
Then, $c(u_1u_{a-2})=c(u_{a-2}u_{a-1})=c(u_{a-1}u_\ell)$.
Thus, there exist joint monochromatic triangles in $(G,c)$, a contradiction. Hence, the assertion follows.

Since $C_{u_1}$ is proper,
according to Lemma \ref{lemma12},
if $|P_{C_{u_1}}^3(u_1)|=2$,
the result follows.
In the following suppose $|P_{C_{u_1}}^3(u_1)|\neq 2$.

In the first case suppose $|P_{C_{u_1}}^3(u_1)|\geq 3$.
Then, $|P_{C_{u_1}}^2(u_1)|=1$;
otherwise, $c(wu_3)=c(u_2u_3)$, a contradiction.
If $c(u_1,P_{C_{u_1}}^1(u_1))=c_w$,
then according to Lemma \ref{lemma11},
the result follows.
If $c(u_1,P_{C_{u_1}}^1(u_1))\neq c_w$,
then $c(u_1u_i)\notin \{c_w,c(u_iu_{i+1})\}$ for $i=4,5$.
Hence, $c(wu_4)=c(wu_3)=c(wu_\ell)$,
which means $|P_C^3(w)|=2$, a contradiction.

In the second case suppose $|P_{C}^3(w)|=1$
and $|P_{C_{u_1}}^3(u_1)|\leq 1$.
When $c(wu_3)\neq c_w$, $C_{u_2}$ is proper.
Since $wu_3u_2u_4\overrightarrow{C}u_1w$ is a cycle of length $\ell+1$ and contains $v_1$,
we have $c(u_2u_4)\in \{c(u_2u_3),c(u_4u_5)\}$.
If $c(u_2u_4)=c(u_4u_5)$, then the result holds.
If $c(u_2u_4)=c(u_2u_3)\neq c(u_4u_5)$,
then $u_2\in W_3(C_{u_2})$ with $|P_{C_{u_2}}^3(u_2)|\geq 2$.
Thus, according to Lemmas \ref{lemma11} and \ref{lemma12}, the result follows.
When $c(wu_3)= c_w$,
from our assertion we know $R_C(w)=P_C^1(w)=V(u_4\overrightarrow{C}u_1)$.
Since for $u_k\in V(u_5\overrightarrow{C}u_\ell)$,
both $wu_3\overrightarrow{C}u_{k-1}u_1u_2u_k\overrightarrow{C}u_\ell w$
and $wu_{k-1}\overleftarrow{C}u_2u_k\overrightarrow{C}u_1$
are of length $\ell+1$ and contain $v_1$,
we have $c(u_2u_k)=c(u_ku_{k+1})$.
The result thus follows.

In the third case suppose $P_{C}^3(w)=\emptyset$
and $|P_{C_{u_1}}^3(u_1)|\leq 1$.
Apparently, if $|P_{C_{u_1}}^3(u_1)|= 1$, the result holds.
In the following suppose $P_{C_{u_1}}^3(u_1)=\emptyset$.
If $u_3\overrightarrow{C}u_2$ is rainbow,
from Proposition \ref{p4},
there is a requested cycle in $(G,c)$.
If $u_4\notin R_C(w)$,
then assuming $c(wu_4)=c(uu_p)$, we have $u_p\in R_C(w)$.
Since $c(u_1u_k)=c(wu_k)$, we have $c(wu_k)\neq c_w$ for $u_k\in P_C^1(w)$.
Since $wu_{2}u_{3}u_1\overleftarrow{C}u_{p}u_{a-1}\overrightarrow{C}u_{p-1}w$
is of length $\ell+1$ and contains $v_1$,
we have $c(u_1u_3)=c(u_1u_\ell)$,
that is, $c(wu_3)=c(wu_\ell)$.
Then, $u_4\overrightarrow{C}u_1$ is rainbow.
According to Proposition \ref{p4},
there is a requested cycle in $(G,c)$.
\end{proof}

\begin{pro}\label{p5}
Let $(G,c)$ be an edge-colored complete graph on $n\geq 3$ vertices such that $\delta^c(G)\geq \frac{n+1}{2}$, and not contain joint monochromatic triangles. For any $w\in W_3(C)$, if $Q_C(w)=P_C^2(w)$, then $(G,c)$ is properly vertex-pancyclic.
\end{pro}

\begin{proof}
Suppose, to the contrary, that there exists a vertex $v$ which is contained in a proper $\ell$-cycle $C$
in $(G,c)$ for some $\ell$ with $4\leq \ell \leq n-1$, but no proper cycle of length $\ell+1$ in $(G,c)$ contains vertex $v$.
According to Lemma \ref{lemma9},
suppose that there is a vertex $w\in W_3(C)$ such that $|P_C^3(w)|\leq 2$.

According to Lemmas \ref{lemma10}, \ref{lemma11}, \ref{lemma12} and \ref{lemma13},
there is a vertex set $V_1$ of size $\ell-4$ such that $|\mathcal{C}(u,V(C)\setminus V_1)\cap \mathcal{C}(u, V_1)|\geq 2$
for $u\in V_1$, and $V_1$ has the $DP_w$.
Then, there is a vertex $u_p\in V_1$ with $d^c_{V(C)}(u_p)\leq \frac{\ell+1}{2}$.
If $W_3(C)=\{w\}$, $d^c_{V(C)}(u_p)\leq \frac{\ell+1}{2}<\frac{n+1}{2}$, a contradiction.
Thus, $|W_3(C)|\geq 2$.
Then, we give $(G[V_1],c)$ a coloring orientation.
Assume $u_p$ is the maximum out-degree in $D(G[V_1])$.

If $W_3(C)=\{w,w'\}$,
then $d^+_D(u_p)=\frac{\ell-5}{2}$; otherwise,  $d^c(u_p)<\ell-5-\frac{\ell-5}{2}+4=\frac{\ell+3}{2}\leq \frac{n+1}{2}$, a contradiction.
Since the average out-degree of $D$ is $\frac{\ell-5}{2}$,  we have
$d^+_D(u)=\frac{\ell-5}{2}$ for $u\in V_1$.
Therefore, $V_1\subseteq P_C^3(w')$; otherwise, the color of $w'u$ is a used color in $\mathcal{C}(u,V(C))$.
Then, there is a vertex whose color degree less that $\frac{n+1}{2}$.
Thus, $P_C^3(w')=u_5\overrightarrow{C}u_\ell$ or $P_C^3(w')=u_3\overrightarrow{C}u_{\ell-2}$;
otherwise, according to Lemma \ref{lemma1},
$(G,c)$ has joint monochromatic triangles.
Then, we might as well suppose $P_C^3(w')=u_5\overrightarrow{C}u_\ell$.
Then, $w'u_3u_4$ is monochromatic.
According to Lemma \ref{lemma9}, we have that
$u_3\in W_3(C_{u_3})$ with $|P_{C_{u_3}}^3(u_3)|\leq 1$
or $u_4\in W_3(C_{u_4})$ with $|P_{C_{u_4}}^3(u_4)|\leq 1$.
(Note that $C_{u_i}=u_{i-1}w'u_{i+1}\overrightarrow{C}u_{i-1}$, $i=3,4$).
Thus, there is a vertex $u_i\in V_1$
such that $\{c(u_3u_i),c(u_4u_i)\}\cap \{c(u_iu_{i+1}),c(u_iu_{i-1})\}\neq \emptyset$.
Then, $d^c(u_i)<\ell-5-\frac{\ell-5}{2}+3=\frac{\ell+1}{2}\leq \frac{n+1}{2}$, a contradiction.

If $W_3(C)=\{w,w',w''\}$,
then we can suppose that $wv_1v_i$ is not monochromatic, $i=2,\ell$. Then $u_p\in P_C^3(w')\cap P_C^3(w'')$;
otherwise, $d^c(u_p)\leq \ell-5-\frac{\ell-5}{2}+1+3=\frac{\ell+3}{2}<\frac{n+1}{2}$, a contradiction.
Thus, $\frac{\ell-5}{2}\leq d^+_D(u_p)\leq \frac{\ell-4}{2}$;
otherwise, $d^c(u_p)\leq \ell-5-\frac{\ell-3}{2}+5=\frac{\ell+3}{2}<\frac{n+1}{2}$, a contradiction.
When $d^c(u_p)=\frac{\ell-5}{2}$,
since the average out-degree of $D$ is $\frac{\ell-5}{2}$,
we have $d^+_D(u)=\frac{\ell-5}{2}$ for $u\in V_1$.
Therefore, $V_1\subseteq P_C^3(w')\cap P_C^3(w'')$.
Thus, $P_C^3(w')\cap P_C^3(w'')=u_5\overrightarrow{C}u_\ell$ or $P_C^3(w')\cap P_C^3(w'')=u_3\overrightarrow{C}u_{\ell-2}$.
Then, there is a vertex $u\in V_1$ such that $d^c(u)<\frac{n+1}{2}$, a contradiction.
When $d^c(u_p)=\frac{\ell-4}{2}$,
there is a distinct vertex $u'$ in $V_1$ such that
$d^+(u')\geq \frac{\ell-5}{2}$.
Thus, $\{u_p,u'\}\subseteq P_C^3(w')\cap P_C^3(w'')$.
According to the claim of Corollary \ref{coro1},
we have $|P_C^3(w')\cap P_C^3(w'')|\geq 3$.
According to Lemma \ref{lemma9},
we might as well have $|P_C^3(w')|\leq 1$ and $w''v_1v_i$ is monochromatic, $i=2,\ell$.
Then, $c(w'u_p)\subseteq \{c(u_pu_{p+1}),c(u_pu_{p-1})\}$.
Thus, $d^c(u_p)\leq \ell-5-\frac{\ell-4}{2}+1+3=\frac{\ell+2}{2}<\frac{n+1}{2}$, a contradiction.
\end{proof}

Eventually, combining Propositions \ref{p1} through \ref{p5}, we get the proof of our main result Theorem \ref{main}.

\end{document}